\documentclass[12pt]{article}
\usepackage{amsmath}
\usepackage{amsfonts}
\usepackage{amssymb}
\usepackage{amscd}
\usepackage{amsthm}

\usepackage{color}

\input xypic

\renewcommand{\P}{{\mathbb P}}
\newcommand{\Z}{{\mathbb Z}}

\newcommand{\C}{{\mathbb C}}
\newcommand{\Cb}{{\mathbb C}}

\newcommand{\F}{{\mathcal F}}

\newcommand{\E}{{\mathcal E}}

\newcommand{\Tc}{{\mathcal T}}

\newcommand{\res}{{\rm res}}
\newcommand{\Pol}{{\rm Pol}}
\newcommand{\Cous}{{\rm Cous}}
\newcommand{\pol}{{\rm pol}}

\newcommand{\Gc}{{\mathcal G}}

\newcommand{\Pc}{{\mathcal P}}

\newcommand{\K}{{\mathcal K}}
\newcommand{\LL}{{\mathcal L}}

\newcommand{\A}{{\mathbf A}}
\newcommand{\Ab}{{\mathbb A}}

\newcommand{\OO}{{\mathcal O}}

\newcommand{\Hom}{{\rm Hom}}
\newcommand{\Homc}{{\mathcal Hom}}

\newcommand{\Ker}{{\rm Ker}}

\newcommand{\Tr}{{\rm Tr}}

\newcommand{\Spec}{{\rm Spec}}

\renewcommand{\dim}{{\rm dim}}

\theoremstyle{plain}
\newtheorem{theor}{Theorem}[section]
\newtheorem{prop}[theor]{Proposition}

\newtheorem{corol}[theor]{Corollary}
\newtheorem{lemma}[theor]{Lemma}

\theoremstyle{remark}
\newtheorem{rmk}[theor]{Remark}
\newtheorem{examp}[theor]{Example}

\theoremstyle{definition}
\newtheorem{defin}[theor]{Definition}
\newtheorem{defin-prop}[theor]{Definition-Proposition}
\newcommand{\toiso}{\xrightarrow{\sim\;}}
\newcommand{\supp}{{\rm supp}}

\title{A polar complex for locally free sheaves}
\author{ Sergey Gorchinskiy and Alexei Rosly\\ \\
\small{Steklov Mathematical Institute, Moscow, Russia}\\
\small{e-mail: {\tt gorchins@mi.ras.ru}}\\
\small{Institute of Theoretical and Experimental Physics}\\
\small{and Institute for Information Transmission Problems, Moscow, Russia}\\
\small{e-mail: {\tt rosly@itep.ru}}}
\date{}

\begin{document}
\maketitle

\begin{abstract}
We construct the so-called polar complex for an arbitrary locally
free sheaf on a smooth variety over a field of characteristic zero.
This complex is built from logarithmic forms on all irreducible
subvarieties with values in a locally free sheaf. We prove that
cohomology groups of the polar complex are canonically isomorphic to
the cohomology groups of the locally free sheaf. Relations of the
polar complex with Rost's cycle modules, algebraic cycles, Cousin complex,
and adelic complex are discussed. In particular, the polar complex is a
subcomplex in the Cousin complex. One can say that the polar complex is a
first order pole part of the Cousin complex, providing a much smaller, but,
in fact, quasiisomorphic subcomplex.
\end{abstract}

\section{Introduction}

In this paper we are going to study certain complexes where the differential is
constructed from a residue map acting on differential forms with logarithmic
singularities. To explain the motivations, let us start from the following classical
fact: a divisor $\sum a_i x_i$ with complex coefficients $a_i$ on a compact Riemann
surface is the residue of a logarithmic $1$-form if and only if $\sum_i a_i=0$. A
higher-dimensional version of this concerns a \mbox{$p$-di\-men\-sional} subvariety $Z$
in a smooth complex algebraic variety $X$ together with a holomorphic
differential form of top degree (that is, a $p$-form) $\alpha\in
H^0(Z,\omega_Z)$\,. The question is then whether there exists a
$(p+1)$-dimensional subvariety $Y$ in $X$ containing $Z$ and a logarithmic top
form $\beta\in H^0(Y,\omega_Y(Z))$ such that $\res\,\beta=\alpha$. More
generally, one can consider formal linear combinations of subvarieties together
with logarithmic top forms on them.

When $X$ is a compact Riemann surface, the sum $\sum_i a_i$ is an element in the
group~$\Cb$ and one can canonically identify $\Cb$ with $H^1(X,\omega_X)$ by
taking integrals of (1,1)-forms over $X$. In the higher-dimensional case it is
shown in \cite{KRT} that the obstruction for the existence of $(Y,\beta)$ as
above is an element in the group $H^{d-p}(X,\omega_X)$, where $X$ is a
smooth complex projective variety and $d=\dim\,X$. The aim of this paper is to
prove an analogous result for smooth varieties over a field of characteristic
zero and forms with coefficients in vector bundles.

Given a smooth variety $X$ and a locally free sheaf $\F$ on $X$, one constructs a
chain complex $\Pol_{\,\bullet}(X,\F)$, called a {\it polar complex}, whose
$p$-chains are finite formal sums of pairs $(Z,\alpha)$, where $Z$ is a
$p$-dimensional irreducible subvariety in $X$ and $\alpha$ is a logarithmic top form
on $Z$ with coefficients in $\F|_Z$ (see Definitions~\ref{defin-polarformsubvar}
and~\ref{defin-polarcomplex}). The boundary map in the polar complex is defined by
the residue morphism on logarithmic forms. The main result of the paper
(Theorem~\ref{theor-main}) is that, for $\dim\,X=d$, there is a canonical isomorphism
\begin{equation}\label{eq-main}
H^{d-p}(X,\omega_X\otimes_{\OO_X}\F)\cong H_p(\Pol_{\,\bullet}(X,\F))\,.
\end{equation}
This theorem answers the above question about obstructions to being one's
residue. On the other hand it gives a new interpretation of cohomology groups of
locally free sheaves in terms of logarithmic forms on subvarieties.

The isomorphism~\eqref{eq-main} has, in particular, the following interpretation when $X$ is projective.
The homology groups of the polar complex,
$H_{p}(\Pol_{\,\bullet}(X,\F))$, admit a canonical pairing with
$H^p(X,\F^\vee)$, the cohomology of a dual to $\F$. This is described as
follows. For a polar $p$-cycle $(Z,\alpha)$, $\dim\,Z=p$, and an element $u\in
H^p(X,\F^\vee)$, we can define $u'\in H^p(Z,\F^\vee|_Z)$ as a restriction of $u$
to $Z$ and form a natural pairing $\langle\alpha,\,u'\rangle$, because~$\alpha$
is an element of $H^0(Z\,,\omega_Z\otimes_{\OO_Z}\F|_Z)$ (we assume here~$Z$
being smooth for simplicity). We claim that if $\alpha=\res\,\beta$ for $(Y,\beta)$ as above, then this pairing vanishes for any $u$. To see this, represent $u$ by a $(p,0)$-form $\nu$ with coefficients in $\F^{\vee}$. Then, the pairing in question is equal to the integral $\int_Z {\rm ev}(\alpha\wedge\nu|_Z)$, where ${\rm ev}$ denotes the natural coupling of forms with coefficient in $\F$ and with coefficients in $\F^{\vee}$, which gives usual forms. This integral equals to $\frac{1}{2\pi i}\int_Y {\rm ev}(\beta\wedge(\bar\partial\nu)|_Y)$ by a higher-dimensional generalization of the Cauchy formula. Finally, the integrand vanishes, as $\nu$ is $\bar\partial$-closed. So, taking a sum over all \mbox{$p$-dimensional} subvarieties~$Z$ in $X$, we obtain a pairing
$H_{p}(\Pol_{\,\bullet}(X,\F))\otimes H^p(X,\F^\vee) \to\C$. The
isomorphism~\eqref{eq-main} amounts to saying that this pairing is non-degenerate. In particular, we prove that if the pairing with a $(Z,\alpha)$ vanishes, then $\alpha$ is one's residue.

Our proof of the relation~\eqref{eq-main} is based on an analogy between
cohomology groups of locally free sheaves on algebraic varieties and singular
cohomology of local systems on manifolds, see Subsection~\ref{subsection-topanalogy}.
We also use Quillen's trick in algebraic $K$-theory as an algebraic analogue of a collar neighborhood in topology.
Let us explain the idea of this construction in the case of trivial coefficients,
$\F=\OO_X$.
Given a \mbox{$p$-dimensional} subvariety $Z\subset X$, let us look for a $(p+1)$-dimensional
subvariety $T\subset X$ such that $T$ contains $Z$ and possesses a morphism $\rho\!:T\to Z$ so that $Z$ is a retract of $T$. Assume also that $Z$ is given by an equation $t=0$ on $T$. Then, we can see that for any regular $p$-form $\alpha$ on $Z$, we can construct a rational $(p+1)$-form $\beta:=\frac{dt}{t}\wedge\rho^*(\alpha)$ on $T$ that satisfies $\res\,\beta=\alpha$. In fact, such $T$ can in general be found only if we work Zariski locally and replace $X$ by a finite cover that has a section over $Z$. This will nevertheless be sufficient for proving that any polar cycle is locally exact. See more details in Section~\ref{section-polarresol}.

Note that methods of the present paper are completely
different from those in \cite{KRT}, where the isomorphism~\eqref{eq-main} is
established for $\F=\OO_X$ and $X$ projective.

One can also find relations of the polar complex with other concepts. In the
first place, the polar complex can be regarded as a Gersten type
complex\footnote{By this we mean a complex whose $p$-th term is naturally
written as a direct sum over $p$-dimensional irreducible subvarieties.}
associated with a certain cycle module over $X$. Recall that cycle modules were
introduced by M.\,Rost in~\cite{Rost}. We prove in
Theorem~\ref{theor-polarresol} that the polar complex is locally exact on $X$
except for its left most term. It is worth noticing that an analogue of
Theorem~\ref{theor-polarresol} is not true for an arbitrary cycle
module over a non-trivial $X$ (different from a point). Nevertheless, the polar
complex for a locally free sheaf $\F$ on $X$ corresponds to a cycle module over
$X$, which, for a non-trivial $\F$ (i.e., different from $\OO_X$)
is non-constant in Rost's terminology, that is to say, this cycle module is
not a pull-back of a one defined over a point. Thus, the polar complex gives a
new example of a (graded component in a) non-constant cycle
module over a variety such that, nevertheless, the corresponding Gersten
complex is locally exact except for its left most term.

It might also be worth mentioning that, the polar complex provides a canonical
construction of classes of algebraic $p$-cycles in groups
$H^{d-p}(X,\Omega^{d-p}_X)$.
On the other hand, the polar complex $\Pol_{\,\bullet}(X,\F)$ is canonically a
subcomplex in the Cousin complex of $\omega_{X}\otimes_{\OO_X}\F$. The
isomorphism~\eqref{eq-main} implies that the polar complex is actually
quasiisomorphic to the Cousin complex. One can say that the polar complex is a
first order pole part of the Cousin complex, providing thus a much smaller, but
still quasiisomorphic subcomplex.

In addition, the polar complex on a smooth complex projective variety fills in a
vacant site in a chain of quasiisomorphisms connecting the rational adelic complex
(see \cite{Par,Bei,Hub}) to the Dolbeault complex, such that each complex in this
chain has an explicit geometric description.

Let us mention one more motivation for the construction of polar complexes with
coefficients in locally free sheaves. It comes from a study of a holomorphic analogue
of the linking number. Such an analogue is supposed to be a certain function of
positions of, say, two complex curves in a three-dimensional complex manifold. In the
paper \cite{A}, Atiyah considered the case of two lines in a complex projective space
$\P^3$. He defined a function (or, rather, a section of a line bundle) on the
configuration space of such pairs and exploited it in a construction of Green's
function for a certain Laplace operator. This function depends holomorphically on the
positions of the lines and develops a pole on the discriminant of intersecting lines,
while the topological linking number is a locally constant function on the
configuration space; more on the analogy between topological and holomorphic objects
see in Subsection~\ref{subsection-topanalogy}. Atiyah suggested to regard this
function as a holomorphic analogue of the classical Gauss linking number in topology.
As the latter can be nicely described in terms of singular chains, one may ask what
should be a holomorphic analogue of singular chains appropriate to the context of the
above holomorphic linking. The polar complex can be also thought of as an answer to
that question. Indeed, one can notice that in the construction of \cite{A} each line
$L\subset \P^3$ is implicitly endowed with a section
$\alpha\in H^0(L,\omega_L\otimes\OO_{\P^3}(2)|_L)$, the latter space being
(non-ca\-no\-ni\-cally) isomorphic to $\C$. Thus, $(L,\alpha)$ is a polar
1-cycle in $\P^3$ with coefficients in~$\OO_{\P^3}(2)$. The holomorphic linking
number for polar cycles in arbitrary smooth varieties with coefficients in
locally free sheaves can be defined in the same way as in \cite{KR,KR2,FT}.

The paper is organized as follows. In Section~\ref{sect-prelim} we recollect general
properties of logarithmic forms and fix some notations. In particular, we prove here
a version of the \mbox{Grauert--Riemenschneider} theorem
(Proposition~\ref{prop-logrational}). Note that the proof is purely algebraic.
Section~\ref{section-mainresult} contains the main definitions and describes the main
result (Definitions~\ref{defin-polarformsubvar}, \ref{defin-polarcomplex} and
Theorem~\ref{theor-main}). We also give some topological motivations in
Subsection~\ref{subsection-topanalogy}. The proof of the main result in
Section~\ref{section-motivic} is done in several steps. Namely, in
Subsection~\ref{subsection-polarshv} we define for each locally free sheaf a certain
subsheaf of abelian groups, which we call a polar sheaf
(Definition~\ref{defin-polarsheaf}). The use of polar sheaves allows us to split the
main theorem into two assertions about polar sheaves (Theorem~\ref{theor-polarresol}
and Theorem~\ref{theor-polarquasiis}). The second assertion says that cohomology
groups of a locally free sheaf coincide with cohomology groups of its polar subsheaf.
Subsections~\ref{section-cyclemod} and~\ref{subsection-geomlem} contain simple
auxiliary lemmas. Quillen's trick is applied to prove Theorem~\ref{theor-polarresol}
for projective varieties in Subsection~\ref{theor-polarresol}.
Theorem~\ref{theor-polarquasiis} for projective varieties is proved in
Subsection~\ref{section-relationcoher}. Finally, reduction of the general case to the
projective one is done in Subsection~\ref{section-reductionproj}. In
Section~\ref{section-ps}, the relations of the polar complex with Rost's cycle
modules, algebraic cycles, Cousin complex, and adelic complex are
discussed.

\bigskip
We are grateful to A.\,Beilinson, A.\,Bondal, Vik.\,S.\,Kulikov,
D.\,Orlov, A.\,N.\,Parshin, C.\,Shramov, and V.\,Vologodsky for
useful conversations and comments.
We are also indebted to anonymous referees for a number of valuable remarks,
which helped us to improve the paper substantially.
During the work on the paper we
enjoyed the hospitality and the excellent working conditions at the
University of Aberdeen, the Institute for Physics and Mathematics of
Universe (IPMU) in Kashiwa, and Forschunginstitut f\"ur Mathematik
(FIM) at the ETH in Z\"urich. The work of S.G. was partially
supported by the grants RFBR 11-01-00145, 12-01-31506, 12-01-3302, 13-01-12420, NSh--2998.2014.1, and by Dmitry Zimin's Dynasty Foundation. The work of A.R. was partially supported by the grants
RFBR 11-01-00962, RFBR-09-01-93106-NCNIL-a and by Ministry of Science and Education. This work was also partially supported by AG Laboratory NRU-HSE, RF government grant, ag. 11.G34.31.0023.

\section{Logarithmic forms}\label{sect-prelim}

Throughout the paper we fix a ground field $k$ of {\it characteristic zero}.
We are not assuming that $k$ is algebraically closed. Nevertheless, this does not
lead to any difficulties of arithmetic kind and neither our results nor the
proofs would be simpler for algebraically closed $k$.

All varieties and morphisms between them are considered over $k$. By a
divisor we always mean a Weil divisor. Given a morphism $f$ between varieties,
$f^{-1}(\cdot)$ denotes the set-theoretical preimage of subvarieties (not the
scheme-theoretical one). If not specified, we use Zariski topology on varieties
and by `local' we mean local in Zariski topology.

For a smooth variety $V$, denote by $\Omega^p_V$ the sheaf of differential
$p$-forms. We also set $\omega_V:=\Omega^d_V$, where~$d$ is the dimension of
$V$. For an irreducible variety $V$ (not necessarily smooth), denote by $k(V)$
the field of rational functions on $V$. Let $\Omega^p_{k(V)}$ be the space of
rational $p$-forms on $V$ and, correspondingly,
$\omega_{k(V)}:=\Omega^d_{k(V)}$. Among rational forms, the logarithmic ones
(e.g.\ \cite[Sect.\ 3.5]{GH}, \cite[Sect.\ 8.2]{V}) will be of most importance
for us. Consider a smooth irreducible variety $V$ and a smooth reduced
irreducible divisor~$W$ in~$V$. A rational form $\alpha\in \Omega^p_{k(V)}$
will be called {\it locally logarithmic along} $W$ if both~$\alpha$ and~$d\alpha$ have at most first order poles along $W$. Equivalently, at the generic
point of $W$ the form $\alpha$ can be written (not uniquely) as
$$
\alpha=\frac{dt}{t}\wedge\beta+\gamma\,,
$$
where $t=0$ is an equation of $W$ at its generic point, while $\beta$ and $\gamma$ are rational
$(p-1)$-form and $p$-form on $V$, respectively, which are regular at the generic point of $W$. If now~$\alpha_1$ and~$\alpha_2$ are forms locally logarithmic along $W$, then the form
$\alpha_1\wedge\alpha_2$ is also locally logarithmic along $W$. Moreover, for
any non-zero rational function $h$ on $V$, the rational 1-form $d\log h=dh/h$ is
locally logarithmic along any $W$.

In the above notations, consider the sheaf
$$
\Omega_V^p(\log W):=\Ker\left(\Omega^p_V(W)\to j_*\Omega^p_W(W)\right),
$$
where $j\!:W\hookrightarrow V$. The sheaf $\Omega_V^p(\log W)$ is a locally free
sheaf on $V$. A form \mbox{$\alpha\in H^0(V\smallsetminus W,\Omega_V^p)$} is locally
logarithmic along $W$ if and only if $\alpha\in H^0(V,\Omega_V^p(\log W))$. It
follows from the exact sequence
$$
0\to j^*\OO_V(-W) \to j^*\Omega^1_V\to \Omega^1_W\to 0
$$
and a little linear algebra with wedge powers
that there is a well defined morphism of sheaves
$j^*\Omega^p_V(\log W)\to \Omega^{p-1}_W$. This gives a {\it residue homomorphism}
$$
\res_{VW}\!:H^0(V,\Omega^p_V(\log W))\to H^0(W,\Omega_W^{p-1}).
$$
Locally along $W$ we have $\res_{VW}(\frac{dt}{t}\wedge\beta)=\beta|_W$.

We call $W\subset V$ a {\it simple normal crossing divisor} if $W=\cup W_i$, where
each $W_i$ is a smooth reduced irreducible divisor in $V$ and the divisors $W_i$ meet
transversely at each point of $V$. Explicitly, for any point $x\in V$, local equations
$t_1,\dots,t_r$ of $W_1,\ldots,W_r$ have linearly independent differentials at $x$,
where $W_1,\ldots,W_r$ are components of $W$ that contain $x$. In this case, let
$\Omega^p_V(\log W)$ be the sheaf of rational differential forms that are locally
logarithmic along all divisors $W_i$ and regular elsewhere on $V$. In particular, we

have $\Omega^d_V(\log W)=\omega_V(W)$, where $d$ is the dimension of $V$. Zariski
locally at any point $x\in V$, a form $\alpha\in H^0(V,\Omega_V^p(\log W))$ can be
written (not uniquely) as
\begin{equation}\label{eq-logexpl}
\alpha=\sum_{1\leqslant i_1<\ldots<i_l\leqslant r}\frac{dt_{i_1}}{t_{i_1}}
\wedge\ldots\wedge
\frac{dt_{i_l}}{t_{i_l}}\wedge \beta_{i_1,\ldots,i_l}\,,
\end{equation}
where $l$ runs over $0,\ldots,p$, while $t_1,\dots,t_r$ are local equations of
$W_1,\ldots,W_r$ at $x$, and~$\beta_{i_1,\ldots,i_l}$ is a regular differential form
at $x$ of degree $p-l$. The sheaf $\Omega^p_V(\log W)$ is a locally free sheaf. In
what follows, by a {\it logarithmic $p$-form} on $V$ we will understand an element of
$H^0(V,\Omega_V^p(\log W))$ for some simple normal crossing divisor $W\subset V$.

The properties of logarithmic forms with a smooth divisor of poles described above
extend to a more general case of simple normal crossing divisors. Namely, the
product of differential forms defines a morphism of sheaves
$$
\Omega^p_V(\log W)\otimes_{\OO_X}\Omega^q_V(\log W)\to\Omega^{p+q}_V(\log W).
$$
Moreover, for any non-zero rational function $h$ on $V$ such that $h$ and $h^{-1}$
are regular on the complement $V\smallsetminus W$, the rational differential form
$d\log h=dh/h$ belongs to $H^0(V,\Omega^{1}_{V}(\log W))$. In the above notations,
for any $i$, denote by $W-W_i$ the union of all irreducible components in $W$ except
$W_i$. For any $\alpha\in H^0(V,\Omega^p_V(\log W))$, there is a residue
$$
\res_{VW_i}(\alpha)\in H^0(W_i\smallsetminus (W_i\cap (W-W_i)),\Omega^{p-1}_{W_i}).
$$
Actually, $W_i':=W_i\cap (W-W_i)$ is a simple normal crossing
divisor in $W_i$ and the rational form $\res_{VW_i}(\alpha)$ belongs
to $H^0(W_i,\Omega^{p-1}_{W_i}(\log W'_i))$. Below we recollect some
further properties of logarithmic forms that will be needed later.
First we show that the logarithmicity of rational differential forms
is preserved by pull-backs.

\begin{lemma}\label{lemma-pullbacklog}
Let $f:V'\to V$ be a dominant morphism between smooth varieties and
suppose $W\subset V$, $W'\subset V'$ are simple normal crossing
divisors such that $f^{-1}(W)\subseteq W'$. Then for any
differential form $\alpha\in H^0(V,\Omega^p_{V}(\log W))$, the
rational differential form $f^*\alpha$ belongs to
$H^0(V',\Omega^p_{V'}(\log W'))$.
\end{lemma}
\begin{proof}
There is a covering of $V$ by open subsets $U$ such that the restriction
$\alpha|_U$ has the shape as in equation \eqref{eq-logexpl}. Then $f^*\alpha$ is
locally logarithmic along any smooth irreducible divisor in $U':=f^{-1}(U)$, because
$f^*\alpha|_{U'}$ is a sum of products of $d\log$'s of rational functions and
regular differential forms. Since $f^*\alpha$ has poles only along components of
a simple normal crossing divisor $W'$, this completes the proof.
\end{proof}

Let $K$ be a finitely generated field over $k$, that is, $K=k(V)$ for some
irreducible variety $V$ over $k$, and consider the space of rational $p$-forms
$\Omega^p_K$\,. Let now $K\subset K'$ be a finite extension of fields. Then, the
trace map $\Tr_{K'/K}:K'\to K$ induces a map on rational differential forms
$$
\Tr_{K'/K}:\Omega^p_{K'}=\Omega^p_K\otimes_K K'\to \Omega^p_K \,, \qquad p\ge 0 \,.
$$
The trace map commutes with the de Rham differential (but does not respect
multiplication of forms) and there is a projection formula involving the trace and
the pull-back maps on rational differential forms. If $f\!:V'\to V$ is a generically
finite dominant morphism of irreducible varieties then we denote the trace map
$\Tr_{k(V')/k(V)}$ on rational differential forms by $f_*$\,. When $f$ is finite and
unramified, and the varieties $V$ and $V'$ are smooth, the trace of a regular form
$\alpha'$ on $V'$ can be explicitly written at a point $x\in V$ as follows:
$$
(f_*\alpha')(x)=\sum_{x'\in
f^{-1}(x)}((df_{x'})^\vee)^{-1}(\alpha'(x')) \,.
$$
For an arbitrary generically finite dominant morphism $f$ between smooth $V'$ and
$V$, this formula determines the trace of a rational differential form on a
non-empty open subset in $V$.

If the morphism $f$ is generically finite and proper one can show that the trace
$f_*\,\alpha'$ of a regular differential form $\alpha'$ on $V'$ is regular on
$V$. To prove the latter\footnote{See \cite{Griff} for an analytic approach.}
let us consider an open subset $U\subset V$ such that $f$ is finite over $U$.
One can assume that the complement $V\smallsetminus U$ is of codimension at
least two in $V$. Therefore, it suffices to show that $f_*\,\alpha'$ is regular
in $U$. Without loss of generality, we can assume further that
$f\!:f^{-1}(U)\to U$ is a Galois covering (possibly, ramified). To show this, recall that any finite morphism is dominated by a Galois covering (that is, there exists a Galois covering of $U$ which factors through $f$). By lifting $\alpha'$ to such a Galois covering one gets then an equivalent problem. Therefore, we may simply assume that $f\!:f^{-1}(U)\to U$ is Galois. Consider now the form $f^*f_*\,\alpha'$ on $f^{-1}(U)$. This
form is equal to a sum of Galois conjugates of $\alpha'$. (This is obviously so
over an open subset of $U$, where $f$ is unramified; the equality for the whole
of $U$ then follows.) The regularity of $\alpha'$ implies now the regularity of
$f^*f_*\,\alpha'$. The desired assertion follows by noticing that
$f^*f_*\,\alpha'$ is regular if and only if $f_*\,\alpha'$ is regular (on $U$
and, hence, on $V$).

The next result says that also the logarithmicity of rational differential forms
is preserved by the trace map.

\begin{lemma}\label{lemma-tracelog}
Let $f:V'\to V$ be a proper generically finite surjective morphism between smooth
varieties and suppose $W\subset V$, $W'\subset V'$ are simple normal crossing
divisors such that $W'\subseteq f^{-1}(W)$. Then for any differential form
$\alpha'\in H^0(V',\Omega^p_{V'}(\log W'))$, the rational differential form
$f_*\alpha'$ belongs to $H^0(V,\Omega^p_{V}(\log W))$.
\end{lemma}
\begin{proof}
A restriction of $f$ defines the following proper morphism
$$
V'\smallsetminus f^{-1}(W)\to V\smallsetminus W.
$$
Since $W'\subseteq f^{-1}(W)$ and $f$ is proper, the form $f_*\alpha'$ is regular on
$V\smallsetminus W$. Let $W_i\subset W\subset V$ be an irreducible component of $W$
and choose an equation $t_i$ of $W_i$ at its generic point. Then the forms
\begin{equation}\label{tfalpha}
t_i\cdot f_*\alpha'=f_*((f^*t_i)\cdot\alpha'),
\end{equation}
\begin{equation}\label{tdfalpha}
t_i\cdot d(f_*\alpha')=t_i\cdot f_*(d\alpha')=f_*((f^*t_i)\cdot d\alpha')
\end{equation}
are regular at the generic point of $W_i$. Therefore the forms $f_*\alpha'$ and $d(f_*\alpha')$ have at
most first order poles along~$W_i$, so $f_*\alpha'$ is locally logarithmic along
$W_i$\,, and this completes the proof.
\end{proof}

For the sake of technical arguments below we shall need to know how the
residue map commutes with pull-back and trace morphisms of differential forms.
Let us consider a dominant morphism between smooth varieties, $f:V'\to V$, and
suppose $W$ is a smooth irreducible divisor in $V$ such that $W':=f^{-1}(W)$ is a
simple normal crossing divisor in~$V'$. Let $\{W_i'\}$ be the set of those
irreducible components of~$W'$ which map dominantly to $W$. Denote by $e_i$ the
ramification index of $f$ along~$W'_i$ and denote by $f_i:W'_i\to W$ the restriction
of $f$. Then for any differential form $\alpha\in H^0(V,\Omega^p_{V}(\log W))$, a
local calculation shows that
\begin{equation}\label{equat-pullback}
\res_{V'W'_i}(f^*\alpha)=e_i\cdot f_i^*(\res_{VW}(\alpha)).
\end{equation}
Suppose in addition that $f$ is proper and generically finite. Then for any
differential form $\alpha'\in H^0(V',\Omega^p_{V'}(\log W'))$, the form $f_*\alpha'$ is logarithmic by Lemma~\ref{lemma-tracelog} and we have
\begin{equation}\label{equat-trace}
\res_{VW}(f_*\alpha')=\sum_{i}f_{i*}(\res_{V'W'_i}(\alpha')).
\end{equation}
Indeed, let $t=0$ be an equation of $W$ at its generic point. Then the $1$-form $f^*\left(\frac{dt}{t}\right)$ is locally logarithmic along each $W'_i$ and its residue to $W'_i$ equals $e_i$. There exists a \mbox{$(p-1)$-form} $\beta'$ on $V'$ which is regular at the generic point of each $W'_i$ and such that for each $i$, we have $\res_{V'W_i'}(\alpha')=e_i\cdot \beta'|_{W'_i}\,$. Then the form $\gamma':=\alpha'-f^*\left(\frac{dt}{t}\right)\wedge\beta'$ has no residues to~$W'_i$, whence this form is regular at the generic point of each $W'_i$. It follows that $f_*\gamma'$ is regular at the generic point of $W$ and by the projection formula, it is enough to show that $(f_*\beta')|_W=\sum_i e_if_{i*}(\beta'|_{W'_i})$. There exist $p$-forms $\beta_j$ on $V$ which are regular at the generic point of $W$ and rational functions $h'_j$ on $V'$ which are regular at the generic point of each $W'_i$ such that $\beta'=\sum_j h'_j\cdot f^*\beta_j$ (compare with the definition of the trace map on forms as given above). Using the projection formula again, we reduce the equation on forms to the following equation on functions: $(f_*h)|_W=\sum_i e_if_{i*}(h|_{W_i})$, where $h$ is a rational function on $V'$ which is regular at the generic point of each $W'_i$ (recall that here $f_*$ and $f_{i*}$ are nothing but the trace maps for the corresponding field extensions). The latter equation can be shown either analytically after the reduction to the case $k=\Cb$, or by using complete local fields associated with divisors $W_i$.

Our next technical lemma follows directly from the Bertini theorem (note that the Bertini
theorem is valid over an arbitrary infinite field as any dense open subset in
the projective space has a rational point):

\begin{lemma}\label{lemma-Bertini}
Let $V$ be a smooth quasi-projective variety, $\varphi\!:V\to\P^N$ a regular
morphism, $W\subset V$ a simple normal crossing divisor. Then for a general
hyperplane $H\subset \P^N$, the divisor\,\footnote{By $\varphi^{-1}(H)$ we denote the
set-theoretical preimage of $H$; in particular, $\varphi^{-1}(H)$ is a reduced
divisor.} $\varphi^{-1}(H)$ does not contain any component of $W$, and
$W\cup \varphi^{-1}(H)$ is a simple normal crossing divisor.
\end{lemma}

We will also need the following form of the Hironaka theorem (see
\cite{Hir}):

\begin{prop}\label{prop-Hironaka}
\hspace{0cm}
\begin{itemize}
\item[(i)]
For an irreducible variety $V$ and a subvariety $Z\subset V$, there exists a proper
birational morphism\footnote{By a birational morphism we mean a regular morphism
between irreducible varieties that has a rational inverse.} $\pi:\widetilde{V}\to V$
such that $\widetilde{V}$ is smooth, $\pi^{-1}(Z)$ is a simple normal crossing
divisor in $\widetilde{V}$, and the morphism $\pi$ is an isomorphism over the
complement \mbox{$V\smallsetminus (V_{sing}\cup Z)$}, where $V_{sing}$ is the
singular locus of $V$.
\item[(ii)]
Let $\varphi:V_1\dasharrow V_2$ be a birational map between complete irreducible
varieties. Then there is a composition of blow-ups at smooth centers $f:V'_1\to V_1$
and a regular morphism $\varphi':V'_1\to V_2$ such that $\varphi\circ f=\varphi'$.
\end{itemize}
\end{prop}
Throughout this article we repeatedly encounter one and the same need of
resolving a rational map. For the sake of further reference, we formulate this
simple corollary of the Hironaka theorem as a separate lemma.
\begin{lemma}\label{lemma-resolmorph}
Let $f:V\to X$ and $f':V'\to X$ be proper morphisms from irreducible varieties
to a variety $X$. Let $g:V'\dasharrow V$ be a rational map such that
$f\circ g=f'$. Then, there exist a smooth irreducible variety $V''$ and proper
morphisms $h:V''\to V$, $h':V''\to V'$ such that $h'$ is birational and $g\circ h'=h$.
\end{lemma}
In other words, for $V$ and $V'$ proper over $X$, a rational map $g:V'\dasharrow V$
over $X$  can be turned into a proper regular morphism by replacing $V'$ with a
suitable birational smooth variety $V''$, which is summarized in the following
commutative diagram:
\begin{equation}\nonumber
\begin{array} {c}

    ~~V''  \\
h' \swarrow ~~~ \searrow h \\
V' ~~~ \stackrel{g}{\dashrightarrow} ~~~ V \\
f' \searrow ~~~ \swarrow f \\
    X
\end{array}
\end{equation}

\begin{proof}
Let $U\subset V'$ be a non-empty open subset such that $g_U:=g|_U$ is regular on $U$.
Consider the embedding of the graph of $g_U$ followed by an open embedding as follows:
$$
\Gamma_{g_U}\subset U\times_X V\subset V'\times_X V.
$$
Take the closure $\bar{\Gamma}_{g_U}$ of $\Gamma_{g_U}$ in $V'\times_X V$. Note
that $\bar{\Gamma}_{g_U}$ contains $\Gamma_{g_U}\cong U$ as an open dense subset
and the projections to the two factors in  $V'\times_X V$ define proper
morphisms from $\bar{\Gamma}_{g_U}$ to $V'$ and $V$. As a smooth $V''$ we can
now pick up a resolution of singularities of~$\bar{\Gamma}_{g_U}$.
\end{proof}

The next result can be considered as a logarithmic version of the
\mbox{Grauert--Riemenschneider} theorem \cite{GR} (in other words, of the assertion
that smooth varieties have rational singularities). Since the authors could not find
any reference for this case, the proof is given.
For a morphism
$\pi\!:\widetilde{X}\to X$ between smooth irreducible varieties of equal dimensions
and a locally free sheaf $\Gc$ on $X$,
we set:
\begin{equation}\label{equat-factorial}
\pi^!\Gc:=\omega_{\widetilde{X}}
\otimes_{\OO_{\widetilde{X}}}\pi^*(\omega_X^{-1}\otimes_{\OO_X}\Gc).
\end{equation}
We will use that $\pi^!$ extends to a functor between derived categories of coherent
sheaves, which is right adjoint to the functor $R\pi_*$ by the Grothendieck duality,
provided that $\pi$ is proper (see \cite{Har}).

\begin{prop}\label{prop-logrational}
Let $\pi\!:\widetilde{X}\to X$ be a proper birational morphism between smooth
irreducible quasi-projective varieties. Let $D\subset X$ and
$\widetilde{D}:=\pi^{-1}(D)$ be simple normal crossing divisors. Then, we have that
$$
R\pi_*\omega_{\widetilde{X}}(\widetilde{D})=\omega_{X}(D) \,.
$$
\end{prop}
\begin{proof}
The Hironaka theorem (see Proposition~\ref{prop-Hironaka}) implies the existence of a
smooth projective closure of all the data in the proposition. The statement of the
proposition is local along $X$, while the fibers of the proper morphism $\pi$ over
$X$ are not changing under taking a projective closure. Therefore, we may assume that
$\widetilde{X}$ and $X$ are projective.

By Proposition \ref{prop-Hironaka}~$(ii)$ applied to the birational map
$\varphi:=\pi^{-1}$, there is a smooth variety $Y$ together with regular birational
proper morphisms $f\!:Y\to X$ and $\widetilde{f}\!:Y\to \widetilde{X}$ such that
$f=\pi\circ\widetilde{f}$ and $f$ is a composition of blow-ups at smooth centers. We
set $C:=f^{-1}(D)$, $C':=\widetilde{f}^*(\widetilde{D})$. By definition of
$\widetilde{f}^{~!}$, we have that
$\widetilde{f}^{~!\;}\omega_{\widetilde{X}}(\widetilde{D})=\omega_Y(C')$. Now, since
$\widetilde{f}^{~!}$ is right-adjoint to $R\widetilde{f}_*$, there is a trace morphism
$$
R\widetilde{f}_*\omega_Y(C')\to \omega_{\widetilde{X}}(\widetilde{D}) \,.
$$
Equation~\eqref{eq-logexpl} implies the existence of the pull-back map on differential
forms, \mbox{$\omega_{\widetilde{X}}(\widetilde{D})\to \widetilde{f}_*\omega_Y(C)$},
because $\widetilde{D}$ is a simple normal crossing divisor. The composition of
this pull-back map with the natural morphisms
$\widetilde{f}_*\omega_Y(C)\to\widetilde{f}_*\omega_Y(C')$, for $C\leqslant C'$,
and $\widetilde{f}_*\omega_Y(C')\to R\widetilde{f}_*\omega_Y(C')$, followed by
the above trace morphism gives us, in particular, a sequence
$$
\omega_{\widetilde{X}}(\widetilde{D})\to R\widetilde{f}_*\omega_Y(C)\to
\omega_{\widetilde{X}}(\widetilde{D})\,,
$$
which is an identity on $\omega_{\widetilde{X}}(\widetilde{D})$ (it is enough to
check this on an open subset in $\widetilde{X}$ over which $\widetilde{f}$ is an
isomorphism). Thus, there is a splitting
$$
R\widetilde{f}_*\omega_Y(C)\cong\omega_{\widetilde{X}}(\widetilde{D})\oplus
N
$$
for a certain object $N$ in the derived category of coherent sheaves
on $\widetilde{X}$. Applying further~$R\pi_*$, we obtain a splitting
$$
Rf_*\omega_Y(C)\cong
R\pi_*\omega_{\widetilde{X}}(\widetilde{D})\oplus R\pi_*N  \,.
$$
On the other hand, we claim that $Rf_*\omega_Y(C)=\omega_X(D)$. To prove this we may
assume that $f$ is a blow-up at a smooth center $Z\subset X$. By the projection
formula, it is enough to consider only those components of $D$ which contain $Z$,
because for the other components, their set-theoretical preimages coincide with their
pull-backs. Then, a direct calculation gives the required result.

Thus, we see that $R\pi_*\omega_{\widetilde{X}}(\widetilde{D})$ is a direct summand
in $\omega_X(D)$, which proves immediate\-ly the proposition, because an invertible
sheaf on a connected variety has no non-trivial direct summands.
\end{proof}

The following corollary will be useful for our purposes.

\begin{corol}\label{corol-directimage}
Let $\pi\!:\widetilde{X}\to X$ be a proper birational morphism between smooth
quasi-projective varieties. Let $D\subset X$ and $\widetilde{D}:=\pi^{-1}(D)$ be
simple normal crossing divisors. Then, for any locally free sheaf $\Gc$ on $X$, and
for any $q\geqslant 0$, we have that
$$
H^{q}(\widetilde{X},\pi^!\Gc(\widetilde{D}))=
H^{q}(X,\Gc(D)).
$$
\end{corol}

\section{Statement of the main result}\label{section-mainresult}

\subsection{The polar complex}\label{subsection-main}

Let $Z$ be an irreducible variety and $\E$ a locally free sheaf on $Z$. We are going
to define a distinguished subset in $\omega_K\otimes_K\E_K$ where $K:=k(Z)$ and
$\E_K$ denotes the space of rational sections of $\E$. In particular, its elements
will be allowed to have poles of the first order at most.

\begin{defin}\label{defin-polarformsubvar}
An element $\alpha\in \omega_K\otimes_K\E_K$ is called {\it polar} if there exist a
smooth variety~$V$, a proper birational morphism $f\!:V\to Z$, a simple normal
crossing divisor $W\subset V$, and a section
$$
\alpha_V\in H^0(V,\omega_V(W)\otimes_{\OO_V}f^*\E)
$$
such that the restriction of $\alpha_V$ to the generic point of $V$ equals $f^*\alpha$.
We shall use the notation
$$
\Pol_Z(\E) = \{{\rm
the~set~of~all~polar~elements~in~}\omega_K\otimes_K\E_K \,\}.
$$
\end{defin}

\begin{rmk}\label{rmk-polarcheck}
It follows from the Hironaka theorem that, for $Z$ and $\E$ as in
Definition~\ref{defin-polarformsubvar}, the set $\Pol_Z(\E)$ is a $k$-vector subspace
in \mbox{$\omega_K\otimes_K\E_K$}. For more details see
Remark~\ref{rmk-polarcyclesubvar} and Proposition~\ref{prop-polarforms}$(i)$ below.
\end{rmk}

For an irreducible reduced divisor $B$ on $Z$, there is a canonical {\it
residue homomorphism}
$$
\res_{ZB}\!:\Pol_Z(\E)\to\Pol_B(\E|_B)
$$
induced by taking residues of logarithmic forms. This homomorphism is defined as
follows. First, if $Z$ is normal, the generic point $x$ in $B$ is smooth both in
$B$ and $Z$ and any form $\alpha\in\Pol_Z(\E)$ has at most first order pole
along $B$ locally at $x$. In this case, let us define $\res_{ZB}(\alpha)$
as the usual residue of $\alpha$ at $B$.
If $Z$ is not normal, we take the normalization $\pi\!:\widetilde Z\to Z$ and define
$\res_{ZB}(\alpha)$ as the sum
$\sum_i \pi_*(\res_{\widetilde{Z}\widetilde{B}_i}(\pi^*\alpha))$,
where~$\widetilde{B}_i$ runs through irreducible components in the preimage of
$B$ in $\widetilde Z$.
By construction, $\res_{ZB}(\alpha)$ is a rational form on $B$ with coefficients
in $\E|_B$, but, in fact, $\res_{ZB}(\alpha)$ is in $\Pol_B(\E|_B)$. By
Remark~\ref{rmk-polarcyclesubvar}, this is a particular case of
Proposition~\ref{prop-polarforms}$(vi)$ proved below.

Let now $X$ be a variety and choose a locally free sheaf $\F$ on $X$. Taking
residue homomorphisms for all possible pairs $B\subset Z$ of irreducible
subvarieties in $X$, we get a homomorphism
$$
\mbox{$\partial:\bigoplus\limits_{Z\in
X_{(p)}}\Pol_Z(\F|_Z)\;\to\bigoplus \limits_{B\in
X_{(p-1)}}\Pol_B(\F|_B)$},
$$
where $X_{(p)}$ denotes the set of irreducible subvarieties in $X$ of dimension
$p$. For short, put
$$
\mbox{$\Pol_p(X,\F):=\bigoplus\limits_{Z\in X_{(p)}}\Pol_Z(\F|_Z)$}\,.
$$

We have that $\partial^2=0$. This statement is Zariski local on $X$,
so, it is enough to consider the case $\F=\OO_X$. When, in addition, $X$ is a
smooth complex projective variety, the groups $\Pol_p(X,\OO_X)$ and the maps
$\partial$ were introduced and studied in~\cite{KR}, where these groups were
denoted by ${\mathcal C}_p$, see Definitions~3.2,~3.7 in~\cite{KR}. Actually, the
condition $k={\mathbb C}$ does not play any role and for an arbitrary $X$
(possibly, non-smooth and non-projective) the only difference in Definition~3.2
from \cite{KR} that should be made is not to assume that $A$ is projective and to
assume that $f$ is proper in notation therein. With these changes,~${\mathcal
C}_p$ agrees with our group $\Pol_p(X,\OO_X)$ for any $X$. Theorem~3.9 in~\cite{KR}
claims that $\partial^2=0$. The proof of this theorem remains correct
after the above changes, thus, we finally get that $\partial^2=0$ in our case.

\begin{defin}\label{defin-polarcomplex}
A {\it polar complex} of a locally free sheaf $\F$ on a variety
$X$ is the complex $\Pol_{\,\bullet}(X,\F)$
with the differential $\partial$. For a chain $\gamma=\oplus\gamma_Z\in\Pol_p(X,\F)$, with
\mbox{$\gamma_Z\in\Pol_Z(\F|_Z)$}, denote by $\supp\,\gamma$ its {\it support}, that is, the
set of irreducible subvarieties \mbox{$Z\in X_{(p)}$} such that $\gamma_Z\ne 0$.
\end{defin}

The main result of the paper is the following statement.

\begin{theor}\label{theor-main}
Let $X$ be a smooth irreducible quasi-projective variety of
dimension $d$ and suppose $\F$ is a locally free sheaf on $X$. Then,
there is a canonical isomorphism
$$
H^{d-p}(X,\omega_X\otimes_{\OO_X}\F)\cong H_{p}(\Pol_{\,\bullet}(X,\F)) \,.
$$
\end{theor}

\begin{rmk}\label{rmk-notationKR}
Let $X$ be a smooth projective variety and $\F=\OO_X$. The polar complex in this
particular case was dealt with in the papers \cite{KR,KRT}, where the groups
$H_p(\Pol_{\,\bullet}(X,\OO_X))$ were denoted as $HP_p(X)$ and called {\it polar
homology groups}. The above theorem was proved in this case in \cite{KRT}; that is,
$H^{d-p}(X,\omega_X)=HP_{p}(X)$ in the notation of \cite{KRT}.
\end{rmk}

\begin{examp}\label{examp-curve}
Let us describe the complex $\Pol_{\,\bullet}(X,\F)$ in the simplest case when
$X$ is a smooth curve over $k$ and $\F$ a locally free sheaf on $X$. Then,
in homological degree one, we have
$$
\Pol_X(\F)=\varinjlim_D H^0(X,(\omega_X\otimes_{\OO_X}\F)(D))\,,
$$
where the direct limit is taken over all reduced effective divisors $D$ on $X$. In
other words, $\Pol_X(\F)$ consists of all rational sections of
$\omega_X\otimes_{\OO_X}\F$ which have only first order poles somewhere on $X$. In
homological degree zero, for each $x\in X$, we find that $\Pol_x(\F|_x)$ equals
$\F|_x$\,, the fiber at $x$ of the vector bundle associated with $\F$. The polar
complex $\Pol_{\,\bullet}(X,\F)$ then takes the form
$$
0\to\Pol_X(\F)\stackrel{\partial}\to\mbox{$\bigoplus\limits_{x\in X}\F|_x$}\to 0\,.
$$
This is a direct limit of the complexes
$$
0\to H^0(X,(\omega_X\otimes_{\OO_X}\F)(D))\stackrel{\partial}\to
\mbox{$\bigoplus\limits_{x\in D}\F|_x$}\to 0
$$
over all reduced effective divisors $i_D:D\hookrightarrow X$. The assertion of
Theorem~\ref{theor-main} is implied in this case by the long exact sequence of
cohomology groups associated with the exact sequence of sheaves on $X$:
$$
0\to\omega_X\otimes_{\OO_X}\F\to(\omega_X\otimes_{\OO_X}\F)(D)\to
i_{D\,*}(\F|_D)\to 0\,.
$$
\end{examp}

\subsection{A topological analogy}\label{subsection-topanalogy}

The result of Theorem~\ref{theor-main} was called in the paper \cite{KRT}
the Polar de Rham Theorem by a reason of the following
analogy with the topological situation (see, for example, \cite{KR} and references therein):

\vspace{-3mm}
$$
\begin{array}{rcl}
\mbox{smooth real manifold}&\longleftrightarrow &\mbox{complex algebraic manifold}\\
d & \longleftrightarrow & \bar\partial \\
\mbox{de~Rham~complex} & \longleftrightarrow & \mbox{Dolbeault~complex} \\
{\rm smooth~functions~or~sections} & \longleftrightarrow &
                              {\rm smooth~functions~or~sections} \\
{\rm locally~constant~functions~or~sections}
                       & \longleftrightarrow &
{\rm holomorphic~functions~or~sections} \\
\mbox{flat~bundles (locally constant sheaves)}     & \longleftrightarrow &
\mbox{holomorphic~bundles (locally free sheaves)} \\
\mbox{orientation}&\longleftrightarrow &\mbox{holomorphic top form}\\
\mbox{orientation sheaf}&\longleftrightarrow &\mbox{canonical sheaf}\\
\mbox{boundary operator}&\longleftrightarrow &\mbox{residue homomorphism}\\
\mbox{Borel--Moore homology with coefficients}&&
\mbox{polar homology with coefficients}\\
\mbox{in a locally constant sheaf}&\longleftrightarrow&
\mbox{in a locally free sheaf}\\
\end{array}
$$
Thus, Theorem~\ref{theor-main} can be viewed on as an analog of the
following fact (see, e.g., \cite[Sect.\ 8.8 A and D of Ch.\ 2]{VF}):
\begin{prop}\label{claim-topol}
Let $M$ be a smooth real manifold of dimension $d$, suppose $\LL$ is a
locally constant sheaf on $M$, and denote by ${\rm or}_M$  the
orientation sheaf on $M$. Then there is a canonical isomorphism
$$
H^{d-p}(M,{\rm or}_M\otimes_{\Z}\LL)\cong H_{p}^{BM}(M,\LL).
$$
\end{prop}

The proof of Theorem~\ref{theor-main} given in Section~\ref{section-motivic} can be considered as an algebro-geometric analogue of the following proof of
Proposition~\ref{claim-topol}. Recall that the \mbox{Borel--Moore} homology on the
right of the above relation can be defined as homology of the complex
$C_{\bullet}(M,\LL)$ of locally finite chains with coefficients in $\LL$, that is, an
element in $C_p(M,\LL)$ is an infinite formal linear combination of singular
$p$-chains with coefficients in $\LL$ such that for any compact subset $K\subset M$,
there are only finitely many terms whose support meets $K$. Consider a homological
complex of flabby sheaves on $M$ given by the formula
$$
\underline{C}\,_{\bullet}(M,\LL)(U):=C_{\bullet}(M,\LL) \ /\
C_{\bullet}(M\smallsetminus U,\LL|_{M\smallsetminus U})
$$
for an open subset $U\subset M$. To conclude one shows that this complex of sheaves
is quasiisomorphic to the sheaf ${\rm or}_M\otimes_{\Z}\LL$ placed in the homological
degree $d$. A natural way to prove the latter is to use collar neighborhoods, which
allows one to represent locally a singular cycle as a boundary. An analog of collar
neighborhoods in algebraic geometry was invented by Quillen in \cite{Q}. This
construction is sometimes called ``Quillen's trick'' and is often used in various
motivic considerations. We also apply this method in our situation, see
Section~\ref{section-polarresol}.

\section{Proof of the main result}\label{section-motivic}

\subsection{Polar sheaves}\label{subsection-polarshv}

By analogy with what was used above in the proof of
Proposition~\ref{claim-topol}, consider the following sheafified
version of the polar complex:

\begin{defin}\label{defin-Gerst}
For a variety $X$, a locally free sheaf $\F$ on $X$, and $p\geqslant 0$, let us
denote by $\underline{\Pol}_{\,p}(X,\F)$ the sheaf of abelian groups on $X$ defined
so that, for an open subset $U\subset X$, we have
\begin{equation} \label{eq-def4.1}
\underline{\Pol}_{\,p}(X,\F)(U):=\Pol_p(X,\F) \ /\
\Pol_p(X\smallsetminus U,\,\F|_{X\smallsetminus U})  \,.
\end{equation}
\end{defin}

More explicitly, we have that
\begin{equation}\label{eq-relativecohom}
\underline{\Pol}_{\,p}(X,\F)(U)=\mbox{$\bigoplus\limits_{{Z\in
X_{(p)}}\atop{Z\cap U\ne \varnothing}}$}\Pol_Z(\F|_Z)\,.
\end{equation}
The latter expression implies that the sheaves $\underline{\Pol}_{\,p}(X,\F)$ are
flabby and form a complex $\underline\Pol_{\,\bullet}(X,\F)$ with a differential
$\partial$ induced by that on $\Pol_{\,\bullet}$ as defined in
Section~\ref{section-mainresult}.

We are going to deal with polar sections of $\Gc=\omega_X\otimes_{\OO_X}\F$\,,
which are distinguished among other rational sections by saying (roughly) that
they have at most first order poles, see Definition~\ref{defin-polarformsubvar}.
Therefore, it looks reasonable to consider also the following subsheaf of
abelian groups, $\Gc_{\pol}$\,, of the sheaf $\Gc$\,. For an open subset
$U\subset X$, the group~$\Gc(U)$ consists of rational sections of $\Gc$ on $X$
which are regular in $U$\,, and we consider a subset of those,
$\Gc_{\pol}(U)$\,, characterized by a certain restriction on their singularities
at $X\smallsetminus U$\,. We wished to say here ``first order poles at
$X\smallsetminus U$\,'', but a precise formulation is as follows.

\begin{defin-prop}\label{defin-polarsheaf}
Let $\Gc$ be a locally free sheaf on a smooth irreducible variety $X$ and
$U\subset X$ an open subset. Let us choose a smooth variety $\widetilde{X}$ and a
proper birational morphism $\pi\!:\widetilde{X}\to X$ such that
\mbox{$\widetilde{D}:=\pi^{-1}(X\smallsetminus U)$} is a simple
normal crossing divisor. Then, the formula (cf.\ eq.\ (\ref{equat-factorial}) in
Section~\ref{sect-prelim})
$$
\Gc_{\pol}(U)=H^0(\widetilde{X},\pi^!\Gc(\widetilde{D}))
$$
defines a sheaf of abelian groups on $X$ which will be called
the {\it polar sheaf} associated with a locally free sheaf $\Gc$.
The definition does not depend on the choice of $\widetilde{X}$.
\end{defin-prop}
\begin{proof}
Consider the sheaf $\K$ on $X$ defined by the following formula
$$
\K:=\Ker\,\mbox{\Large(}\,\underline{\Pol}_{\:d}(X,\omega_X^{-1}\otimes_{\OO_X}\Gc)
\xrightarrow{\partial}\underline{\Pol}_{\:d-1}(X,\omega_X^{-1}\otimes_{\OO_X}\Gc)
\mbox{\Large)} \,,
$$
where $d$ is the dimension of $X$. Consider an open subset $U\subset X$ and an element $\alpha\in\K(U)$. By Definition~\ref{defin-polarformsubvar} and the Hironaka theorem, there is a smooth variety $\widetilde{X}$, a
proper birational morphism $\pi\!:\widetilde{X}\to X$, and a simple normal crossing divisor $W\subset\widetilde{X}$ such that \mbox{$\widetilde{D}:=\pi^{-1}(X\smallsetminus U)$} is a union of some of the components of $W$ and the rational form $\pi^*\alpha$ on $\widetilde{X}$ belongs to $H^0(\widetilde{X},\pi^!\Gc(W))$. Moreover, since $\partial\alpha=0$, we have that $\res_{\widetilde{X}W_i}(\pi^*\alpha)=0$ for any irreducible component $W_i\subset W$ whose intersection with $\pi^{-1}(U)$ is non-empty. Therefore, $\pi^*\alpha\in H^0(\widetilde{X},\pi^!\Gc(\widetilde{D}))$. This implies that
there is a canonical isomorphism $\K(U)\cong \Gc_{\pol}(U)$. Therefore the
definition of the group $\Gc_{\pol}(U)$ is correct and $\Gc_{\pol}$ is a
sheaf.
\end{proof}

\begin{rmk}\label{pol=H0}
\hspace{0cm}
\begin{itemize}
\item[(i)]
The sheaf $\Gc_{\pol}$ is a subsheaf of abelian groups in $\Gc$
and is not a subsheaf of \mbox{$\OO_X$-modules}. Indeed, multiplying elements in
$\Gc_{\pol}(U)$ by an arbitrary regular function~$h$ on an open subset $U\subset
X$, we may clearly break the condition from Definition~\ref{defin-polarsheaf},
as~$h$ may have a pole of higher order along $D$.
\item[(ii)]
Given an open embedding $U\subset X$, the sheaf $(\Gc_{\pol})|_U$ is a subsheaf
of abelian groups
in the sheaf $(\Gc|_U)_{\pol}$\,, but these two are not equal in general.
Indeed, one has \mbox{$H^0(U,(\Gc|_U)_{\pol})=\Gc(U)$}\,, whereas
$H^0(U,(\Gc_{\pol})|_U)$ consists of those sections in $\Gc(U)$ which have at
most first order poles at $X\smallsetminus U$ (in the precise sense of
Definition~\ref{defin-polarsheaf}).
\item[(iii)]
The functor $\Gc\mapsto \Gc_{\pol}$ is exact in a locally free sheaf $\Gc$
(left exactness is obvious, while right exactness follows from the Serre vanishing
theorem). Similarly, the functors $\Pol_Z(\F|_Z)$ and $\Pol_{\,\bullet}(X,\F)$ are
exact in a locally free sheaf $\F$ on $X$.
\end{itemize}
\end{rmk}

\begin{theor}\label{theor-polarresol}
Let $X$ be a smooth irreducible quasi-projective variety of dimension $d$ and $\F$ a
locally free sheaf on $X$. Then, the complex of sheaves
$\underline\Pol_{\,\bullet}(X,\F)$ on $X$ is a flabby resolution for the sheaf
$(\omega_{X}\otimes_{\OO_X}\F)_{\pol}$, that is, the following sequence of sheaves is
exact:
$$
0\to(\omega_{X}\otimes_{\OO_X}\F)_{\pol}\to
\underline\Pol_{\:d}(X,\F)\xrightarrow{\partial}
\underline\Pol_{\:d-1}(X,\F)\xrightarrow{\partial}\ldots \xrightarrow{\partial}
\underline\Pol_{\:0}(X,\F)\to 0  \,.
$$
\end{theor}

\begin{theor}\label{theor-polarquasiis}
Let $X$ be a smooth irreducible quasi-projective variety and $\Gc$ a locally free
sheaf on $X$. Then the canonical injective morphism of sheaves $\Gc_{\pol}\to\Gc$
induces isomorphisms of cohomology groups:
$$
H^p(X,\Gc_{\pol})\toiso H^p(X,\Gc), \qquad p\geqslant 0.
$$
\end{theor}

Theorem~\ref{theor-polarresol} and Theorem~\ref{theor-polarquasiis} are proved first
for a projective $X$ in Section~\ref{section-polarresol} and
Section~\ref{section-relationcoher}, respectively. A reduction of both theorems to
the projective case is done in Section~\ref{section-reductionproj}. All together this
implies Theorem~\ref{theor-main}.

\subsection{Auxiliary tools}\label{section-cyclemod}

To prove Theorem~\ref{theor-polarresol} and Theorem~\ref{theor-polarquasiis} we
will use the following properties of the polar complex.

\begin{prop}\label{prop-cyclepolar}
For a locally free sheaf $\F$ on a variety $X$ the following is true:
\begin{itemize}
\item[(i)]
Let $f\!:Y\to X$ be a proper morphism. Then, there exists a canonical morphism of
complexes
$$
f_*\!:\Pol_{\,\bullet}(Y,f^*\F)\to\Pol_{\,\bullet}(X,\F)
$$
compatible with the trace of differential forms.
\item[(ii)]
For $f\!:Y\to X$ as above, let $U\subset Y$ be an open subset. Then,
the quotient complex
$$
\Pol_{\,\bullet}(Y,f^*\F)\ /\ \Pol_{\,\bullet}(Y\smallsetminus U,f^*\F|_{Y\smallsetminus U})
$$
canonically depends only on the morphism $f|_U\!:U\to X$\,.
\item[(iii)]
The projection $\pi\!:X\times\P^1\to X$ induces a
quasiisomorphism of complexes
$$
\pi_*:\Pol_{\,\bullet}(X\times\P^1,\pi^*\F)\to\Pol_{\,\bullet}(X,\F).
$$
\item[(iv)]
Let $E$ be a vector bundle on $X$ and $\P(E)$ a projectivization of
$E$. Then, the projection $\pi\!:\P(E)\to X$ induces a
quasiisomorphism of complexes
$$
\pi_*:\Pol_{\,\bullet}(\P(E),\pi^*\F)\to\Pol_{\,\bullet}(X,\F).
$$
\item[(v)]
Let $\pi\!:\widetilde{X}\to X$ be the blow-up at a center $R\subset X$ such
that $R$ is a local complete intersection in $X$. Let $Z$ be any closed
subvariety in $X$ and $\widetilde{Z}:=\pi^{-1}(Z)$ be the preimage of $Z$ in
$X$. Then, the canonical morphism
$$
\pi_*\!:\Pol_{\,\bullet}(\widetilde{Z},\pi^*\F|_{\widetilde{Z}})\to \Pol_{\,\bullet}(Z,\F|_Z)
$$
is a quasiisomorphism.
\end{itemize}
\end{prop}
\begin{proof}
The proofs of $(i)$ and $(iii)$ are essentially
the same as of Theorem 3.8 in the
paper \cite{KR} (see also Remark 3.12 therein) and as of Lemma~3.5
in \cite{KRT}, respectively. The morphism $f_*$ is constructed as follows. Let $W\subset Y$ be an irreducible subvariety of dimension $p$ and let $\beta\in \Pol_W(f^*\F|_W)$ be a polar form so that the pair $(W,\beta)$ gives an element in $\Pol_p(Y,f^*\F)$. If the morphism $f|_W\!:W\to X$ is finite, that is, if \mbox{$\dim\,f(W)=\dim\,W$}, then we put
$$
f_*(W,\beta):=(Z,\alpha)\in\Pol_p(X,\F)\,,
$$
where $Z=f(W)$ and $\alpha:=\Tr_{k(W)/k(Z)}(\beta)$. Otherwise, we put~\mbox{$f_*(W,\beta):=0$}.

To prove $(ii)$, note first that for any irreducible variety $Z$ and a proper
morphism $g\!:Z\to X$\,, the group $\Pol_Z(g^*\F)$ depends only on the
birational class of $g$, in other words, it depends only on the morphism
$\Spec(k(Z))\to X$\,, the restriction to the generic point. This can be shown
with the help of Lemma~\ref{lemma-resolmorph}. The assertion $(ii)$ then follows
from the relation (compare to the equations (\ref{eq-def4.1},
\ref{eq-relativecohom})):
$$
\Pol_p(Y,f^*\F) \ /\
\Pol_p(Y\smallsetminus U,\,f^*\F|_{Y\smallsetminus U}) \cong
\mbox{$\bigoplus\limits_{{Z\in
Y_{(p)}}\atop{Z\cap U\ne \varnothing}}$}\Pol_Z(f^*\F|_Z)\,.
$$

Turning now to $(iv)$, let us first consider the case of a trivial bundle
$\pi\!:X\times\P^n\to X$ and use induction on $n$ as follows. The base of
induction is $(iii)$. To deduce $n$ from $n-1$, notice that the varieties
$X\times \P^n$ and $X\times\P^{n-1}\times\P^1$ have open subsets that are
isomorphic to $X\times\Ab^n$. The corresponding complements are $X\times\P^{n-1}$
and $X\times T$, where
\mbox{$T:=(\P^{n-1}\times\{\infty\})\cup (\P^{n-2}\times\P^1)$}.
Using $(ii)$, we obtain the isomorphism of complexes
$$
\Pol_{\,\bullet}(X\times
\P^n,\pi^*\,\F)\,/\,\Pol_{\,\bullet}(X\times\P^{n-1},\pi^*\,\F)\cong
\Pol_{\,\bullet}(X\times
\P^{n-1}\times\P^1,\rho^*\,\F)\,/\,\Pol_{\,\bullet}(X\times T,\rho^*\,\F)\,,
$$
where $\rho\!:X\times\P^{n-1}\times\P^1\to X$ is the natural projection. By this
isomorphism and the induction hypothesis, it suffices to prove that
$$
\Pol_{\,\bullet}(X\times T,\rho^*\,\F)\to \Pol_{\,\bullet}(X\times
\P^{n-1}\times\P^1,\rho^*\,\F)
$$
is a quasiisomorphism. (Here we use that an embedding of complexes is a
quasiisomorphism if and only if the quotient is acyclic.) By $(iii)$, it is
enough to show that the morphism $f:X\times T\to X\times\P^{n-1}$ induces a
quasiisomorphism
$$
f_*\!:\Pol_{\,\bullet}(X\times T,\rho^*\,\F)\to \Pol_{\,\bullet}(X\times \P^{n-1},\rho^*\,\F)\,.
$$
Put $U:=X\times \P^{n-1}\smallsetminus X\times\P^{n-2}$. Since $f$ restricts to an isomorphism $f^{-1}(U)\to U$, while the restriction of $f$ to $f^{-1}(X\times \P^{n-2})$ is the projection $X\times
\P^{n-2}\times \P^1\to X\times \P^{n-2}$, we conclude by $(ii)$ applied to $U$ and
$f^{-1}(U)$, and $(iii)$ applied to the latter
projection.

For the general case, $\pi\!:\P(E)\to X$, we use Noetherian induction, that is,
we assume that the desired quasiisomorphism is proved for all closed subsets
$S\varsubsetneq X$. Let $U\subset X$ be a non-empty open subset where the
initial bundle is trivial, $\P(E|_U)\cong U\times\P^n$\,. Let $S=X\smallsetminus U$
and let $\pi_S:\P(E|_S)\to S$ be the natural projection. Then, by $(ii)$ and
$(iv)$ for a trivial bundle, we have a  quasiisomorphism
$$
\Pol_{\,\bullet}(\P(E),\pi^*\,\F) \ /\
\Pol_{\,\bullet}(\P(E|_S),\pi_S^*(\F|_S)) \to
\Pol_{\,\bullet}(X,\F) \ /\
\Pol_{\,\bullet}(S,\F|_S)\,.
$$
Thus, we conclude by Noetherian induction and the $5$-lemma.

Finally, $(v)$ is implied by $(ii)$ and $(iv)$, because $\pi$ is an isomorphism
over $Z\smallsetminus (Z\cap R)$ and the restriction of $\pi:\widetilde{Z}\to Z$
to $Z\cap R$ is a projectivization of a vector bundle on~$Z\cap R$.
\end{proof}

Up to now we considered irreducible subvarieties in a given variety $X$ and certain
rational forms of top degree on them with coefficients in a locally free sheaf $\F$
on $X$ (see Definition~\ref{defin-polarcomplex}). For the sake of the proof of
Theorem~\ref{theor-polarresol} it is convenient to work with irreducible varieties
that map to $X$ not necessarily birationally to their image and certain rational
forms of not only top degree on them. Thus, Definition~\ref{defin-polarform} below
can be thought of as an extension of the notion of polar elements in
$\omega_K\otimes_K\F_K$, $K=k(Z),\;Z\subset X$, given in
Definition~\ref{defin-polarformsubvar}, to the case when $Z$ is replaced by an
irreducible variety $V$ together with a morphism $V\to X$, and rational top forms are
replaced by forms of arbitrary degree $q$ in $\Omega^q_K$\,, $K=k(V)$\,. The
requirement of first order poles is then to be replaced by the requirement that the
$q$-forms have logarithmic singularities.

\begin{defin}\label{defin-fieldover}
For a variety $X$, we say that $(K,\varphi)$ is a {\it field over $X$} if $K$ is
a field and there is a given morphism \mbox{$\varphi\!:\Spec(K)\to X$} such that
there exists an irreducible variety $V$ with $k(V)=K$ and a morphism of
varieties $f\!:V\to X$ that agrees at the generic point of $V$ with the given
morphism $\varphi$. We also say that $(K',\varphi')$ is an {\it extension} of
$(K,\varphi)$ over $X$ if $K\subset K'$ and the following diagram is
commutative:
\begin{equation}\nonumber
\begin{array} {c}
 \Spec(K')  ~ \longrightarrow ~ \Spec(K) \\
 \varphi'\searrow ~~~ \swarrow\varphi   \\
      X
\end{array}
\end{equation}
\end{defin}

In other words, for a given $X$, a field $(K,\varphi)$ over $X$ is a class of
birationally equivalent pairs \mbox{$(V,~f\!:V\to X)$}. In fact, one can always choose a representative $(V,f)$ such that~$V$ is smooth and $f$ is proper:

\begin{lemma}\label{lemma-resolproj}
Given a field $(K,\varphi)$ over a variety $X$, there exists a smooth
irreducible variety $V$ with $k(V)=K$ and a proper morphism $f:V\to X$ that
agrees at the generic point of $V$ with the given morphism $\varphi$.
\end{lemma}
\begin{proof}
There exist an affine variety $U$ with $k(U)=K$ and a morphism $f_U:U\to X$ that
agrees at the generic point of $U$ with the given morphism $\varphi$. Let
$\Gamma_{f_U}$ be the graph of~$f_U$. Consider the embedding $U\subset \P^n$ for
some $n$. Then,
$$
\Gamma_{f_U}\subset U\times X\subset \P^n\times X.
$$
Take the closure $\bar{\Gamma}_{f_U}$ of $\Gamma_{f_U}$ in $\P^n\times X$. Note that
$\bar{\Gamma}_{f_U}$ contains $\Gamma_{f_U}\cong U$ as an open dense subset and the
projection to $X$ defines a proper (even projective) morphism
$\bar{\Gamma}_{f_U}\to X$. We can now get a smooth $V$ by resolving the singularities
of $\bar{\Gamma}_{f_U}$.
\end{proof}

\begin{defin}\label{defin-polarform}
Let $X$ be a variety, $\F$ a locally free sheaf on $X$, and $(K,\varphi)$ a field
over~$X$. We define $M_q(\F)(K,\varphi)$ as a set of all elements
$\alpha\in\Omega^q_K\otimes_K\varphi^*\F$ such that there exists a collection
$(V,f,W,\alpha_V)$, where $V$ is a smooth irreducible variety with $k(V)=K$,
$f\!:V\to X$ is a proper morphism that agrees at the generic point of $V$ with the
given morphism $\varphi$, $W\subset V$ is a simple normal crossing divisor,
$$
\alpha_V\in H^0(V,\Omega^q_V(\log W)\otimes_{\OO_V}f^*\F),
$$
and $\alpha$ is the restriction of $\alpha_V$ to the generic point of $V$.
\end{defin}

\begin{rmk}\label{rmk-polarcyclesubvar}
The definition of $M_q$ (Definition~\ref{defin-polarform})
is similar to the definition of polar elements (Definition~\ref{defin-polarformsubvar}), but concerns differential forms not necessarily of top degree and morphisms $f\!:V\to X$ that are not necessarily birational to their image.
Nevertheless, if $Z$ is an irreducible subvariety of dimension $p$ in~$X$, and
\mbox{$\varphi_Z\!:\Spec(k(Z))\to X$} is the natural morphism, then we have
$$
M_p(\F)(k(Z),\varphi_Z)=\Pol_Z(\F|_Z)\,.
$$
\end{rmk}

\begin{prop}\label{prop-polarforms}
Let $\F$ be a locally free sheaf on a variety $X$.
\begin{itemize}
\item[(i)]
For a field $(K,\varphi)$ over $X$, the set $M_q(\F)(K,\varphi)$ is a $k$-vector
subspace in \mbox{$\Omega^q_K\otimes_K\varphi^*\F$}.
\item[(ii)]
If $(K',\varphi')$ is an extension of $(K,\varphi)$ over $X$ (see
Definition~\ref{defin-fieldover}) then the pull-back map
$$
\Omega^q_K\otimes_K\varphi^*\F\to \Omega^q_{K'}\otimes_{K'}\varphi'^*\F
$$
sends $M_q(\F)(K,\varphi)$ to $M_q(\F)(K',\varphi')$.
\item[(iii)]
If $(K',\varphi')$ is an extension of $(K,\varphi)$ over $X$, such that
$K\subset K'$ is a finite extension then the trace map
$$
\Omega^q_{K'}\otimes_{K'}\varphi'^*\F\to\Omega^q_K\otimes_K\varphi^*\F
$$
sends $M_q(\F)(K',\varphi')$ to $M_q(\F)(K,\varphi)$.
\item[(iv)]
For a field $(K,\varphi)$ over $X$, the product of differential forms defines an
associative graded ring structure on
$\bigoplus\limits_{q\geqslant 0}M_q(\OO_X)(K,\varphi)$ and a graded module structure
over this ring on $\bigoplus\limits_{q\geqslant 0}M_q(\F)(K,\varphi)$.
\item[(v)]
For a field $(K,\varphi)$ over $X$, the image of the homomorphism
$d\log\!:K^*\to\Omega^1_K$ is contained in $M_1(\OO_X)(K,\varphi)$.
\item[(vi)]
Given an irreducible variety $V$, a proper morphism $f\!:V\to X$, and
an irreducible divisor $D\subset V$, there is a residue homomorphism
$$
\res_{VD}\!:M_q(\F)(k(V),\varphi)\to M_{q-1}(\F)(k(D),\psi),\quad q\geqslant 1,
$$
where $\varphi$ and $\psi$ are the restrictions of $f$ to the generic point of
\,$V$ and $D$, respectively. If a collection $(V,f,W,\alpha_V)$ is as in Definition~\ref{defin-polarform} and $D$ is an irreducible component of $W$, then the above map $\res_{VD}$ applied to the corresponding element $\alpha\in M_q(\F)(k(V),\varphi)$ is nothing but the usual residue of the logarithmic form $\alpha_V$ as defined in Section~\ref{sect-prelim}.
\item[(vii)]
Given an irreducible variety $V$, a morphism $f\!:V\to X$, an irreducible subvariety
$Z\subset V$ of codimension two, and an element $\alpha\in M_q(\F)(k(V),\varphi)$
with $\varphi$ being the restriction of $f$ to the generic point of $V$, we have
$$
\sum_D(\res_{DZ}\circ\res_{VD})(\alpha)=0
$$
where the sum is taken over all irreducible divisors $D$ in $V$ that contain $Z$.
\item[(viii)]
Suppose $g:X'\to X$ is a proper morphism. For a field $(K,\varphi')$ over $X'$
we set $\varphi:=g\circ\varphi'$. We then have a canonical isomorphism
$$
M_q(\F)(K,\varphi)\cong M_q(g^*\F)(K,\varphi').
$$
\end{itemize}
\end{prop}
\begin{proof}
Note that Lemma~\ref{lemma-pullbacklog} and Lemma~\ref{lemma-tracelog} can be
easily extended to the case of forms with coefficients in a locally free sheaf
on the target variety (denoted by $V$ there). Below we shall use this version of
Lemma~\ref{lemma-pullbacklog} and Lemma~\ref{lemma-tracelog}.

It is clear that, for a field $(K,\varphi)$ over $X$, the subset
$$
M_q(\F)(K,\varphi)\subset \Omega^q_K\otimes_K\varphi^*\F
$$
is closed under multiplication by elements of $k$. Further, consider two
elements $\alpha_1,\alpha_2\in\nolinebreak M_q(\F)(K,\varphi)$
and the corresponding collections
$(V_1,f_1,W_1,\alpha_{V_1})$, $(V_2,f_2,W_2,\alpha_{V_2})$ as in
Definition~\ref{defin-polarform}. Since $k(V_1)=k(V_2)=K$, there is a fixed
birational morphism \mbox{$g:V_1\dasharrow V_2$}. By
Lemma~\ref{lemma-resolmorph}, there is a smooth variety $V$ with proper
birational morphisms $h_i:V\to V_i$, $i=1,2$, such that
$f_1\circ h_1=f_2\circ h_2$.  Let $f:V\to X$ be any of the latter compositions.
By the Hironaka theorem, changing $V$ birationally, we may assume that
$W:=h_1^{-1}(W_1)\cup h_2^{-1}(W_2)$ is a simple normal crossing divisor in $V$. By
Lemma~\ref{lemma-pullbacklog}, the forms $h_i^*\alpha_{V_i}$, $i=1,2$, belong to
$H^0(V,\Omega^q_V(\log W)\otimes_{\OO_V}f^*\F)$. The restriction of the form
$$
h_1^*\alpha_{V_1}+h_2^*\alpha_{V_2}\in H^0(V,\Omega^q_V(\log W)\otimes_{\OO_V}f^*\F)
$$
to the generic point of $V$ is equal to
$\alpha_1+\alpha_2\in\Omega^q_K\otimes_K\varphi^*\F$, which shows $(i)$.

Consider an extension $(K',\varphi')$ of $(K,\varphi)$ over $X$, a form $\alpha$
in $M_q(\F)(K,\varphi)$, and a corresponding collection $(V,f,W,\alpha_V)$ as in
Definition~\ref{defin-polarform}. Let us denote by $\alpha'$ the pull-back of
$\alpha$. Our aim is to show that the rational form  $\alpha'\in
\Omega^q_{K'}\otimes_{K'} \varphi'^*\F$ belongs in fact to
$M_q(\F)(K',\varphi')$. By Lemma~\ref{lemma-resolproj}, there exist a smooth
irreducible variety $V'$ with $k(V')=K'$ and a proper morphism $f':V'\to X$ that
agrees at the generic point of $V'$ with the given morphism $\varphi'$. The
extension $(K',\varphi')$ of $(K,\varphi)$ gives us a rational map
$g:V'\dasharrow V$ over $X$ with $f\circ g=f'$. By Lemma~\ref{lemma-resolmorph},
we may assume now that $g$ is regular. By further  changing $V'$ birationally,
we may assume according to the Hironaka theorem that $W':=g^{-1}(W)$ is a
simple normal crossing divisor in $V'$. By Lemma~\ref{lemma-pullbacklog}, the
pull-back $\alpha_{V'}:=g^*\alpha_V$ gives an element in
$H^0(V',\Omega^q_{V'}(\log W')\otimes_{\OO_{V'}}f'^*\F)$. Since the restriction
of $\alpha_{V'}$ to the generic point of $V'$ is equal to $\alpha'$, the
assertion $(ii)$ is proved.

Suppose now that the extension $K\subset K'$ is finite, consider a form
$\alpha'$ in $M_q(\F)(K',\varphi')$ and a corresponding collection
$(V',f',W',\alpha_{V'})$ as in Definition~\ref{defin-polarform}. Let us denote
by~$\alpha$ the trace of $\alpha'$. Our aim is to show that the rational form
$\alpha\in \Omega^q_K\otimes_K \varphi^*\F$ belongs in fact to
$M_q(\F)(K,\varphi)$. By Lemma~\ref{lemma-resolproj}, there is a smooth
irreducible variety $V$ with $k(V)=K$, a proper morphism $f:V\to X$ that agrees
at the generic point of $V$ with the given morphism $\varphi$. The finite
extension $K\subset K'$ defines a rational map $g:V'\dasharrow V$ such that
$f\circ g=f'$. By Lemma~\ref{lemma-resolmorph}, we may assume that $g$ is
regular. The morphism~$g$ sends the divisor $W'\subset V'$ to a subvariety
$g(W')$ in $V$, which, however, does not have to be a simple normal crossing
in $V$. We may now apply the Hironaka theorem to the pair $g(W')\subset V$ and
get a proper birational morphism $\pi:\widetilde{V}\to V$ such that
$\widetilde{W}:=\pi^{-1}(g(W'))$ is a simple normal crossing divisor in
$\widetilde{V}$. Simultaneously, this gives us a rational map
\mbox{$\tilde{g}:V'\dasharrow \widetilde{V}$} and by
Lemma~\ref{lemma-resolmorph}, we may assume that this morphism is regular. By
abuse of notation, we shall now write $V$ instead of $\widetilde{V}$, $g$
instead of $\tilde{g}$,  and $W$ instead of $\widetilde{W}$. Thus, we
constructed a proper morphism $g:V'\to V$ such that  $f\circ g=f'$, $g$
corresponds at the generic point of $V'$ to the given morphism of fields over
$X$, and we have that $W'\subseteq g^{-1}(W)$. By Lemma~\ref{lemma-tracelog},
the trace $\alpha_V:=g_*\alpha_{V'}$ gives an element in
$H^0(V,\Omega^q_{V}(\log W)\otimes_{\OO_{V}}f^*\F)$. Since the restriction of
$\alpha_V$ to the generic point of $V$ is equal to $\alpha$, the assertion
$(iii)$ is proved.

Further, $(iv)$ follows by the same argument as in $(i)$ and the fact that
logarithmic forms are closed under products, provided the union of their poles is a
simple normal crossing divisor.

The assertion $(v)$ follows from the fact that $d\log$ applied to a non-zero rational
function gives a logarithmic form, provided that the support of the divisor of that
function is a simple normal crossing divisor. The latter can be achieved by the Hironaka
theorem.

Consider $V$, $f$, and $D$ as in $(vi)$. Note that $V$ and $D$ are not necessarily
smooth.  Take a form $\alpha$ in $M_q(\F)(k(V),\varphi)$ and a corresponding
collection $(V',f',W',\alpha_{V'})$  as in Definition~\ref{defin-polarform}. This
means in particular that $k(V')\cong k(V)$. By Lemma~\ref{lemma-resolmorph} we may
assume that the latter isomorphism corresponds to a birational morphism $g:V'\to V$.
Changing $V'$ birationally if necessary, we can make $g^{-1}(D)\cup W'$ be a simple
normal crossing divisor in $V'$. Let $\{D'_i\}$ be the set of all components of
$g^{-1}(D)$ that map dominantly to $D$. Denote by $g_i:D'_i\to D$ the restriction of
$g$ to $D'_i$ and by $\psi'_i:\nolinebreak\Spec(k(D'_i))\to X$ the restriction of $f'$
to the generic point of $D'_i$. For each $i$, the usual residue
$\res_{V'D'_i}(\alpha_{V'})$ at $D_i'$ defines an element $\beta_i$ in
$M_{q-1}(\F)(k(D'_i),\psi'_i)$. By $(iii)$ and $(i)$, the rational form
$\beta:=\sum_i g_{i*}(\beta_i)$ belongs to $M_{q-1}(\F)(k(D),\psi)$. We set now
$\res_{VD}(\alpha):=\beta$. Since the residue map commutes with the pull-back
and the trace map (cf.\ eqs.~\eqref{equat-pullback} and \eqref{equat-trace} in
Section~\ref{sect-prelim}; note that the ramification indices are $e_i=1$ in the
present case) the form $\beta$ does not depend on the choices made. This
explains~$(vi)$.

The proof of $(vii)$ goes without any difference with the proof of Theorem 3.9
in~\cite{KR}.

Finally, $(viii)$ is shown as follows. We have an obvious map
$M_q(g^*\F)(K,\varphi')\to M_q(\F)(K,\varphi)$. Let us construct the inverse
map. Let $(V,f,W,\alpha_V)$ be as in Definition~\ref{defin-polarform}. Then,
$\varphi'$ induces a rational map $h:V\dasharrow X'$. By
Lemma~\ref{lemma-resolproj} and the Hironaka theorem, there exists a proper
birational morphism $\pi:V'\to V$ from a smooth irreducible variety~$V'$, a
morphism $f':V'\to X'$, and a simple normal crossing divisor $W'\subset V'$ such
that $f'=h\circ \pi$ and $\pi^{-1}(W)\subseteq W'$. Then $(V,f,W,\alpha_V)$ in
$M_q(\F)(K,\varphi)$ is sent to~$(V',f',W',\pi^*\alpha_V)$.
\end{proof}

\subsection{A geometric lemma}\label{subsection-geomlem}

In the next subsection we shall use the following simple geometric
fact (it seems that there is no suitable analogue of that fact when
$X$ is not projective):

\begin{lemma}\label{lemma-projection}
Let $X$ be a smooth irreducible projective variety of dimension $d$
and $x\in X$ a point there. Let $Z\subset X$, $W\subset X$ be two
closed subvarieties such that $Z\ne X$ and any irreducible component
of $W$ has codimension at least two in $X$. Then, there exist a
blow-up $f\!:\widetilde{X}\to X$ at a finite set not containing $x$
and not intersecting $Z$ and $W$, and a surjective morphism
$\pi\!:\widetilde{X}\to \P^{d-1}$ such that $\pi$ is smooth at
$\widetilde{x}:=f^{-1}(x)$\,, the restriction
$\varphi:=\pi|_{\widetilde{Z}}$ to $\widetilde{Z}:=f^{-1}(Z)$ is
finite, and
$$
\widetilde{W}\cap\pi^{-1}(\pi(\widetilde{x}))\subseteq\{\widetilde{x}\}
\,,
$$
where $\widetilde{W}:=f^{-1}(W)$.
\end{lemma}

\begin{rmk}
The last condition means that we are trying to minimize the intersection of
$\widetilde{W}$ and the fiber of $\pi$ through $\widetilde{x}$. If $x\notin W$, then
we can achieve that they do not meet at all, otherwise, we can achieve that they
intersect only at $\widetilde{x}$.
\end{rmk}

\begin{proof}
Choose an embedding $X\subset \P^N$. Denote by ${\mathbb T}_xX$ the
projective tangent space to~$X$ at the point $x$, that is, the unique linear subspace in
$\P^N$ of dimension $d$ tangent to~$X$ at~$x$. Denote by $C_xW$ the cone over
$W$ with the vertex at $x$, that is, the Zariski closure of the union of all
lines that pass through $x$ and meet $W$. By the Bertini theorem, there exists
a linear subspace $L\subset \P^N$ of dimension $N-d$ such that $L$ meets $X$
transversely,~$L$ does not contain $x$, $L\cap {\mathbb T}_xX$ is a point,
$L\cap Z=\varnothing$, and $L\cap C_xW=\varnothing$. Let $f\!:\widetilde{X}\to X$
be the blow-up of $X$ at the finite set $L\cap X$. The projection
$\pi_L\!:\P^N\dasharrow\P^{d-1}$ with the center at $L$ defines then a morphism
$\pi\!:\widetilde{X}\to\P^{d-1}$ where $\pi=\pi_L|_{_X}\circ f$, which satisfies
all the properties needed. In particular, the restriction
$\pi|_{\widetilde{Z}}=\pi_L|_{_Z}$ is finite, as the center~$L$ does not
meet~$Z$,~see~\cite[Theorem 7 in Ch. I, \S 5]{Shaf}.
\end{proof}

\subsection{Algebraic collar neighborhood}\label{section-polarresol}

In this subsection we prove Theorem~\ref{theor-polarresol} when $X$
is a smooth irreducible {\it projective} variety.

\begin{proof}[Proof of Theorem \ref{theor-polarresol} for the projective case]
As it was noticed in the proof of Proposition~\ref{defin-polarsheaf}, there is a
canonical isomorphism of sheaves
$$
(\omega_X\otimes_{\OO_X}\F)_{\pol}\toiso
\Ker\,(\,\underline{\Pol}_{\:d}(X,\F)\to\underline{\Pol}_{\:d-1}(X,\F)) \,.
$$
Thus, it remains to show that the complex of sheaves
$\underline\Pol_{\,\bullet}(X,\F)$ is acyclic
(that is to say, locally, every cycle is exact).
Let $x\in X$ be an arbitrary
point and $\alpha\in \Pol_p(X,\F)$ be a chain.
In general, $\partial\alpha$ is a sum over irreducible $p$-dimensional
subvarieties in $X$. Let us denote
by $\partial_x\alpha$ the sum of those summands in $\partial\alpha$ whose support
contains $x$. Suppose that $\partial_x\alpha=0$. We have to show that there
exists a chain $\beta\in \Pol_{p+1}(X,\F)$ such that $\partial_x\beta=\alpha$.
Equivalently, this means that having $x\notin\supp\,\partial\alpha$, we want to
find a chain $\beta$ such that $x\notin \supp\,(\alpha-\partial\beta)$\,.

Let \mbox{$Z:=\supp\,\alpha$} and suppose $\{Z_i\}$ are the irreducible components of $Z$,
so that $\alpha=\oplus\alpha_i$ with $\alpha_i\in \Pol_{Z_i}(\F|_{Z_i})$. We may
assume that $Z_i$ contains $x$ for each~$i$.

\smallskip
\begin{rmk}\label{rmk-collar}
The idea of the proof is to use an algebraic version of a {\it
collar neighborhood}. Suppose a subvariety $T\subset X$ is a
``collar neighborhood'' of $Z$ Zariski locally at~$x$. By this we
mean that $Z\subset T$, $\dim\,T=\dim\,Z+1=p+1$, there is a
``contraction'' $\rho\!:T\to Z$, and $T$ is smooth at $x\in Z$ along
the normal direction to $Z$\,. If $\F=\OO_X$, then one could take
$\beta=\frac{dt}{t}\wedge \rho^*(\alpha)$, where $t$ is an equation
of $Z$ in $T$ locally at $x$. The problem is that such a collar
neighborhood $T$ does not always exist. An idea of how to overcome
this problem was found by Quillen in \cite[Theorem 5.11]{Q}. The
main point is to use a certain finite surjective map onto $X$ which
has a section over $Z$, that is, to solve the problem replacing
Zariski topology by Nisnevich topology.
\end{rmk}

First of all, by the local nature of the problem, we may replace $X$ by its
blow-up at a finite set not containing $x$. Hence by
Lemma~\ref{lemma-projection}, we may assume that there exists a morphism
$\pi\!:X\to \P^{d-1}$ such that $\pi$ is smooth at $x$\,, $\varphi:=\pi|_Z$ is
finite, and $W\cap \pi^{-1}(\pi(x))\subseteq\{x\}$, where $W$ is the union of
poles of forms $\alpha_i$ and singularities of~$Z_i$ over all $i$\,. Consider
the scheme $Y'=Z\times_{\P^{d-1}}X$\,, so that we have
$Z\xleftarrow{\,\rho}Y'\xrightarrow{\psi}X$\,. Denote by $Y$ the union of all
irreducible components of $Y'_{red}$ which contain an irreducible component of
$\sigma(Z)$, where $\sigma\!:Z\to Y'$ is a section of $\rho$ induced by the
embedding $Z\hookrightarrow X$. We thus have the diagram
$$
\begin{array}{ccc}
Y&\stackrel{\psi}\longrightarrow&X \\
\sigma\uparrow\downarrow\rho&&\downarrow\lefteqn{\pi}\\
Z&\stackrel{\varphi}\longrightarrow&\P^{d-1}
\end{array}
$$
Note that the morphism $\psi$ is finite, while the morphism $\rho\!:Y\to Z$ is
smooth at the finite set $S:=\psi^{-1}(x)$, as $\pi$ is smooth at $x$\,. In particular, $\rho$ is smooth at $\sigma(x)\in S$. Since all $Z_i$ contain $x$, we see that all $\sigma(Z_i)$ contain $\sigma(x)$. This implies that the set $\{Y_i\}$ of irreducible components of $Y$ is bijective with the set $\{Z_i\}$ of irreducible
components of $Z$ and each morphism $\rho\!:Y_i\to Z_i$ is surjective. In addition, locally at $x\in Z$, the fibers of the morphism~$\rho$ are irreducible curves and all the $Y_i$ have the same dimension~$p+1$. Denote by $\chi_i:\Spec(k(Y_i))\to Y$ the natural morphism for each irreducible
component~$Y_i$ of $Y$.

Since $\rho$ is smooth at $S$, there exists a local equation $t=0$ for the section $\sigma(Z)$ in~$Y$ locally at $S$. By Proposition~\ref{prop-polarforms}\,$(ii),(iv),(v),(viii)$
for each $i$, the rational differential form
\mbox{$\gamma_i:=\frac{dt}{t}\wedge \rho^*\alpha_i$} on $Y_i$ belongs to
$M_{p+1}(\rho^*\F_Z)(k(Y_i),\chi_i)=\Pol_{Y_i}((\rho^*\,\F_Z)|_{Y_i})$, where~$\rho^*\F_{Z}$ is the pull-back to $Y$ of $\F_{Z}$\,, the restriction to $Z$ of the
locally free sheaf $\F$ on $X$\,. Consider the chain
$$
\gamma:=\oplus\gamma_i\in \Pol_{p+1}(Y,\rho^*\F_Z)\,.
$$
For a point $y\in S$ different from $\sigma(x)$, each $\gamma_i$ is regular at $y$,
because $y\notin \sigma(Z)$ and \mbox{$\rho(y)\notin W$}.
Since $\sigma^*(\rho^*\F_Z)=\F_Z$, Proposition~\ref{prop-cyclepolar}\,$(i)$
gives us the morphism
$$
\sigma_*\!:\Pol_{\,\bullet}(Z,\F_Z)\to \Pol_{\,\bullet}(Y,\rho^*\F_Z)\,,
$$
whence we obtain an element $\sigma_*\alpha$ in $\Pol_{\,\bullet}(Y,\rho^*\F_Z)$.
We claim that $\partial_{\sigma(x)}\gamma=\sigma_*\alpha$. Indeed, for each $i$,
let $\partial_x\alpha_i=\oplus\eta_j$, where $j$ runs over irreducible
components of $\supp\,\partial_x\alpha_i$. It was shown above that locally at $x\in Z$ the fibers of the morphism~$\rho$ are irreducible curves. In addition, for each $j$, the support of $\eta_j$, $\supp\,\eta_j$, contains $x$, whence $\rho^{-1}(\supp\,\eta_j)$ is an irreducible variety of dimension $p$. Applying
Proposition~\ref{prop-polarforms}\,$(ii)$, we obtain the pull-back $\rho^*\eta_j$ of~$\eta_j$ from $\supp\,\eta_j$ to $\rho^{-1}(\supp\,\eta_j)$. Since $\res_{Y_i,\sigma(Z_i)}(\frac{dt}{t}\wedge\rho^*\alpha_i)=\sigma_*\alpha_i$, we then have that
$$
\mbox{$\partial_{\sigma(x)}\gamma_i=
\sigma_*\alpha_i\oplus\Big(\oplus_j(\frac{dt}{t}\wedge\rho^*\eta_j)\Big)$}\,.
$$
Taking the sum over all $i$ and using that $\partial_x\alpha=0$, we obtain the desired equality.
So, we see that $\partial\gamma=\sigma_*\alpha$ locally at $S$ on $Y$,
or, generalizing our notation, we write $\partial_S\gamma=\sigma_*\alpha$.

Now, from the chain $\gamma\in\Pol_{p+1}(Y,\rho^*\F_Z)$, we are going to construct
a suitable chain in $\Pol_{p+1}(Y,\psi^*\F)$ and, then, take its push-forward by
$\psi_*$ to $X$. This will give us the desired polar chain $\beta$ in
$\Pol_{p+1}(X,\F)$. For this aim, let us consider a locally free sheaf
$\Homc(\rho^*\F_Z,\psi^*\F)$ on $Y$ and a coherent sheaf
$$
\K:=\Ker\left(\Homc(\rho^*\F_Z,\psi^*\F)\to
\sigma_*\sigma^*\Homc(\rho^*\F_Z,\psi^*\F)\right).
$$
By the Serre vanishing theorem, there exists a very ample invertible
sheaf $\LL$ on $Y$ such that
$$
H^1(Y,\K\otimes_{\OO_Y}\LL)=0.
$$
In this case, the map
$$
H^0(Y,\Homc(\rho^*\F_Z,\psi^*\F\otimes_{\OO_Y}\LL))\to
H^0(Z,\sigma^*\Homc(\rho^*\F_Z,\psi^*\F\otimes_{\OO_Y}\LL))
$$
is surjective. Equivalently, the map
\begin{equation}\label{eq-sigma}
\sigma^*:\Hom(\rho^*\F_Z,\psi^*\F\otimes_{\OO_Y}\LL)\to
\Hom(\F_Z,\F_Z\otimes_{\OO_Z}\LL|_Z)
\end{equation}
is surjective, where we used that
$$
\sigma^*(\rho^*\F_Z)=\sigma^*(\psi^*\F)=\F_Z \,.
$$

Let $D\subset Y$ be the divisor of a general regular section of
$\LL$, and so $\LL\cong \OO_Y(D)$. Then, we may assume that $D\cap
S=\varnothing$ and $\sigma^{-1}(D)$ is a Cartier divisor on $Z$. The
surjectivity of the map~\eqref{eq-sigma} implies that there exists a
morphism of sheaves $\Phi:\rho^*\F_Z\to\psi^*\F(D)$ such that
$\sigma^*\Phi$ coincides with the embedding
$\Upsilon\!:\F_Z\hookrightarrow \F_Z(\sigma^{-1}(D))$. The following
chain on~$Y$,
$$
\Phi(\gamma)=\oplus\Phi(\gamma_i)\in\Pol_{p+1}(Y,\psi^*\F(D)) \,,
$$
satisfies $\partial(\Phi(\gamma))=\Phi(\partial\gamma)$, because the polar
complex is functorial in locally free sheaves. Recall that $\gamma$ was
constructed so that $\partial_S\gamma=\sigma_*\alpha$. Hence, we
have \mbox{$\Phi(\partial_S\gamma)=\Phi(\sigma_*\alpha)$}. Since
the push-forward map from Proposition~\ref{prop-cyclepolar}\,$(i)$ is functorial
in locally free sheaves on the target, we have that
$\Phi(\sigma_*\alpha)=\sigma_*((\sigma^*\Phi)(\alpha))$. Since
$\sigma^*\Phi=\Upsilon$, we have that
$\partial_S(\Phi(\gamma))=\sigma_*(\Upsilon(\alpha))$.

On the other hand, note that the space
$\Pol_{Y_i}(\psi^*\F|_{Y_i})$ consists of such elements in
$$
\Pol_{Y_i}(\psi^*\F(D)|_{Y_i})\subset
\omega_{k(Y_i)}\otimes_{k(Y_i)}(\psi^*\F(D))_{k(Y_i)}=
\omega_{k(Y_i)}\otimes_{k(Y_i)}(\psi^*\F)_{k(Y_i)}
$$
that have at most a first order pole along $D$. Consequently, by the Bertini
theorem (cf.\ Lemma~\ref{lemma-Bertini}), we may assume that $\Phi(\gamma_i)$
belongs to $\Pol_{Y_i}(\psi^*\F|_{Y_i})$, whence $\Phi(\gamma)\in\nolinebreak
\Pol_{p+1}(Y,\psi^*\F)$ with $\partial_S\Phi(\gamma)=\sigma_*\alpha$.
Therefore, by Proposition~\ref{prop-cyclepolar}\,$(i)$, the chain
$$
\beta=\psi_*(\Phi_Y(\gamma))\in \Pol_{p+1}(X,\F)
$$
satisfies $\partial_x\beta=\alpha$. This completes
the proof.
\end{proof}

\begin{rmk} The construction of the spaces of polar elements,
$\Pol_Z(\F|_Z)$, uses $X$ globally and these spaces change after we replace $X$
by an open subset. This is why, in the second part of the proof of
Theorem~\ref{theor-polarresol}, we can not use that, on a Zariski neighborhood of
$S$ in $Y$, the locally free sheaves $\rho^*\F_Z$ and $\psi^*\F$ become
isomorphic. To overcome this problem, we twist $\psi^*\F$ by a very ample
invertible sheaf $\LL$ on $Y$.
\end{rmk}

\subsection{A comparison with cohomology of locally free sheaves}
\label{section-relationcoher}

In this section we prove Theorem~\ref{theor-polarquasiis} for the projective
case. Thus, everywhere in this subsection we assume that $X$ is a smooth
irreducible projective variety, while $\Gc$ is a locally free sheaf on $X$.

\begin{defin}\label{defin-smallsubset}
An open subset $U\subset X$ is called {\it $\Gc$-small} if there exists a smooth
variety~$\widetilde{X}$ together with a proper birational morphism
$\pi\!:\widetilde{X}\to X$ such that $\widetilde{D}:=\pi^{-1}(X\smallsetminus
U)$ is a simple normal crossing divisor in $\widetilde{X}$ and for all $p>0$, we
have $H^{p}(\widetilde{X},\pi^!\Gc(\widetilde{D}))=0$.
\end{defin}

\begin{rmk}\label{rmk-smallsubset}
\hspace{0cm}
\begin{itemize}
\item[(i)]
Suppose $U\subset X$ is a $\Gc$-small subset and choose {\it any} proper birational
morphism $\tilde\pi\!:\widetilde{X}'\to X$ such that
$\widetilde{D}':=\tilde\pi^{-1}(X\smallsetminus U)$ is a simple normal crossing
divisor in $X'$. Then by Lemma~\ref{lemma-resolmorph} and
Corollary~\ref{corol-directimage}, for all $p>0$, we have that
$H^{p}(\widetilde{X}',\pi^!\Gc(\widetilde{D}'))=0$.
\item[(ii)]
It follows from (i) that if $U\subset X$ is $\Gc$-small and the complement
$D:=X\smallsetminus U$ is a simple normal crossing divisor, then
$H^p(X,\Gc(D))=0$ for all $p>0$.
\item[(iii)]
Let $U\subset X$ be a $\Gc$-small subset and $\pi\!:\widetilde{X}\to X$  a
proper birational morphism. Then it follows from (i) that the open subset
$\pi^{-1}(U)\subset\widetilde{X}$ is $\pi^!\Gc$-small.
\end{itemize}
\end{rmk}

We will use the following notation. Given a finite open covering $\{U_i\}$, $i\in I$,
of an open subset $U\subset X$ and a non-empty subset $S\subset I$ we define
$U_S:=\cap_{i\in S} U_i$\,, $D_S:=X\smallsetminus U_S$\,, while, for
$S=\varnothing$\,, we set $U_{\varnothing}:=U$ and
$D_{\varnothing}:=X\smallsetminus U$\,.

First we establish a relation between \v Cech cohomology of a polar sheaf
and cohomology of the underlying locally free sheaf.

\begin{lemma}\label{lemma-Cech}
Let $U\subset X$ be an open subset and suppose $U=\cup_{i\in
I}U_{i}$ is a finite open covering such that, for any non-empty
subset $S\subset I$, the open subset $U_S$ is $\Gc$-small and, for
any subset $S\subset I$ (including $S=\varnothing$), the subvariety
$D_S$ is a simple normal crossing divisor in $X$. Then, for any
$p\geqslant 0$, we have
$$
\check{H}^p(\{U_i\},\Gc_{\pol}|_U)=H^p(X,
\Gc(D_{\varnothing})),
$$
where $\check{H}$ denotes the \v Cech cohomology groups associated with an open covering.
\end{lemma}
\begin{proof}
For each $i\in I$, denote by $D_i'$ the union of all irreducible components of $D_i$
that are not contained in $D_{\varnothing}$. {\it A priori}\,,
$\cap_i D_i'\subset D_{\varnothing}$, because $U=\cup_i U_i$. We claim that, in fact,
$\cap_i D_i'=\varnothing$. Indeed, suppose $Z$ is an irreducible component of the
intersection $\cap_i D'_i$. Since the divisor $\cup_i D_i'$ is a simple normal
crossing one, the codimension of $Z$ in $X$ equals the number $c$ of irreducible
components in $\cup_i D_i'$ that contain $Z$. On the other hand,
$Z\subset D_{\varnothing}$, so there are at least $c+1$ irreducible components in the
simple normal crossing divisor $\cup_i D_i$ that contain $Z$, provided that
$D_{\varnothing}$ is non-empty. Hence the codimension of~$Z$ in $X$ is at least
$c+1$, whence the contradiction.

Consider the following Koszul type complex of locally free sheaves on $X$:
$$
\K^{\bullet}=\mbox{$\{0\to \bigoplus\limits_{i\in I}\OO_X(D'_i)\to
\bigoplus\limits_{\{i,j\}\subset I}\OO_X(D'_i\cup
D'_j)\to\ldots\to\ldots \to\OO_X(D'_I)\to 0\}$},
$$
where $\K^l:=\oplus_{S\subset I, |S|=l+1}\OO_X(D'_S)$\,, $0\le l\le |I|-1$, and
$D'_S:=\cup_{i\in S}D'_i$ for each subset $S\subset I$. The differential in this
complex is given by an alternating sum over the natural embeddings
$\OO_X(D'_S)\hookrightarrow\OO_X(D'_T)$ for $S\subset T\subset I,\;  |T|=|S|+1$.
Note that all divisors~$D'_S$ are reduced. Since $\cap_i D_i'=\varnothing$, the
only non-zero cohomology sheaf of the complex of sheaves $\K^{\bullet}$ sits in
degree zero and equals $\OO_X$. Let us form a tensor product of $\K^{\bullet}$
with the sheaf $\Gc(D_{\varnothing})$, which yields now a resolution for
$\Gc(D_{\varnothing})$. Recall that by condition of the lemma,
$D_{\varnothing}+D'_S=D_S$ and $U_S$ is $\Gc$-small for any $S\subset I$,
$S\neq\varnothing$. Hence by Remark~\ref{rmk-smallsubset}(ii),
$H^p(X,\Gc(D_S))=0$ for all $p>0$, that is, the sheaves in the resolution
$\K^{\bullet}\otimes_{\OO_X}\Gc(D_{\varnothing})$ are acyclic and we have:
$$
H^p(X,\Gc(D_{\varnothing}))=
H^p(H^0(X,\K^{\bullet}\otimes_{\OO_X}\Gc(D_{\varnothing}))).
$$
It remains only to notice that the right hand side coincides by
definition with the \v Cech cohomology groups $\check{H}^p(\{U_i\},
\Gc_{\pol}|_U)$.
\end{proof}

We want to prove now that \mbox{$\check{H}^p(\{U_i\},\Gc_{\pol})= H^p(X,\Gc_{\pol})$}
for any $\Gc$-small open \mbox{covering} $\{U_i\}$ of $X$. This will follow once we
prove that the sheaf $\Gc_{\pol}$ is acyclic on $\Gc$-small open subsets.

\begin{lemma}\label{prop-acycl}
Let $U\subset X$ be an open $\Gc$-small subset. Then for any $p>0$, we have
$$
H^{p}(U,\Gc_{\pol}|_U)=0.
$$
\end{lemma}
\begin{proof}
We follow the idea of the proof of a well-known theorem of H.\,Cartan on
acyclicity of coverings (see, e.g.,~\cite[Section 1.4.5]{Dan}). We use induction
on $p\geqslant 1$. Consider an element $\alpha\in H^p(U,\Gc_{\pol}|_U)$. Any
cocycle is locally exact. Thus, there exists a finite open covering
$U=\cup _{i\in I} U_i$ such that $\alpha|_{U_i}=0$ for all $i\in I$. By Theorem
\ref{theor-polarresol} for the projective case and equation~\eqref{eq-def4.1} in
Section~\ref{subsection-polarshv}, we have an isomorphism
$$
H^{p}(U,\Gc_{\pol}|_U)\cong H_{d-p}(\Pol_{\bullet}(X,\F)\,/\,\Pol_{\bullet}(Z,\F|_Z))\,,
$$
where $d$ is the dimension of $X$, $Z:=X\smallsetminus U$, and
$\F:=\omega_X^{-1}\otimes_{\OO_X}\Gc$. Therefore, by the Hironaka theorem and
Proposition~\ref{prop-cyclepolar}\,$(v)$, we may suppose that every
$D_i=X\smallsetminus U_i$, $i\in I$, is a divisor and $D_I=\cup_{i\in I} D_i$ is
a simple normal crossing divisor in $X$. In particular, this implies that for
any subset $S\subset I$, the divisor $D_S$ is a simple normal crossing divisor
in~$X$.

Let $\LL$ be a very ample invertible sheaf on $X$ such that for any
finite subset $S\subset I$ and for all $p>0$, $m>0$, we have that
$$
H^p(X,\Gc(D_S)\otimes_{\OO_X}\LL^{\otimes m})=0 \,.
$$
By Lemma~\ref{lemma-Bertini}, there exist $d+1$ regular sections of $\LL$ with
irreducible divisors $E_j$, $j=1,\dots,d+1$ such that $\cap_j E_j=\varnothing$,
none of $E_j$ is contained in $D_I$, and the divisor $D_I\cup (\cup_j E_j)$ is a
simple normal crossing divisor in $X$. In particular, this implies that for
any subsets $S\subset I$ and $T\subset \{1,\ldots,d+1\}$, the divisor
$D_S\cup (\cup_{j\in T}E_j)$ is a simple normal crossing divisor in $X$.

Consider the covering $U=\cup V_{ij}$, where the open sets
$V_{ij}:=U_i\smallsetminus E_j$ are $\Gc$-small by construction as well as their
intersections. By ascending induction on $p$, it follows from the \v Cech spectral
sequence that there is an exact sequence
$$
0\to\check{H}^p(\{V_{ij}\},\Gc_{\pol}|_U)\to H^p(U,\Gc_{\pol}|_U)\to
\oplus_{ij}H^p(V_{ij}, \Gc_{\pol}|_{V_{ij}})\,.
$$
By construction, the covering $\{V_{ij}\}$ of $U$ satisfies the
condition of Lemma~\ref{lemma-Cech}. Hence the left term equals
$H^p(X,\Gc(D_{\varnothing}))$, where
$D_{\varnothing}=X\smallsetminus U$. By
Remark~\ref{rmk-smallsubset}(ii), since~$U$ is $\Gc$-small, the left
term vanishes. On the other hand, the right arrow sends any
\mbox{$\alpha\in H^p(U,\Gc_{\pol}|_U)$} to zero. Therefore, $\alpha=0$ and
this gives the needed result.
\end{proof}

We are now able to prove Theorem~\ref{theor-polarquasiis} for the projective case.

\begin{proof}[Proof of Theorem~\ref{theor-polarquasiis} for the projective case]
We are going to apply Lemma~\ref{lemma-Cech} with $U=X$ and, thus,
$D_{\varnothing}=\varnothing$.
By the same application of Lemma~\ref{lemma-Bertini} as in the proof of
Lemma~\ref{prop-acycl}, there exists an open covering $\{U_i\}$ of $X$ that
satisfies the conditions of Lemma~\ref{lemma-Cech}. Combining Lemma~\ref{lemma-Cech}
and Lemma~\ref{prop-acycl}, we get the needed result.
\end{proof}

\begin{rmk}\label{rmk-openpolar}
It follows from Lemma~\ref{lemma-Cech} and the proof of
Theorem~\ref{theor-polarquasiis} for a smooth projective $X$ that for an open subset
$U\subset X$ with $D:=X\smallsetminus U$ being a simple normal crossing divisor in
$X$ and for all $p\geqslant 0$, we have that
$$
H^p(U,\Gc_{\pol}|_U)=H^p(X,\Gc(D)).
$$
Then, by setting $\Gc:=\omega_X\otimes_{\OO_X}\F$ and applying
Theorem~\ref{theor-polarresol} for the projective case and
Definition~\ref{defin-Gerst}, we obtain that
$$
H^{d-p}(X,\omega_X\otimes_{\OO_X}\F(D)) =
H_{p}\left(\ \Pol_{\,\bullet}(X,\F)\,/\,\Pol_{\,\bullet}(D,\F|_D)\ \right)\,.
$$
\end{rmk}

\subsection{A reduction to the projective case}\label{section-reductionproj}

In this section we prove Theorem~\ref{theor-polarresol} and
Theorem~\ref{theor-polarquasiis} in full generality, that is, when $X$ is a smooth
irreducible quasi-projective variety. As a matter of fact, this more general case can
be reduced to the projective case.

The following fact was explained to the authors by D.\,Orlov.

\begin{lemma}\label{lemma-Orlov}
Let $X$ be a quasi-projective variety and $\F$ a locally free sheaf on $X$. Then
there exists a projective variety $\bar X$, an open embedding with dense image
$X\hookrightarrow \bar X$, and a locally free sheaf $\bar \F$ on $\bar X$ such that
$\bar \F|_X\cong \F$. If $X$ is smooth, then one can additionally require that
$\bar{X}$ is smooth and the complement $D:=\bar X\smallsetminus X$ is a simple
normal crossing divisor.
\end{lemma}
\begin{proof}
Let $\OO_X(1)$ be a very ample sheaf on $X$ corresponding to a locally closed
embedding $X\hookrightarrow \P^m$ for some $m$ and let $r$ be the rank of $\F$. For
some natural numbers $n$ and $N$, there is a surjective morphism
$\OO_X^{\oplus n}\to \F(N)$ of sheaves on $X$. The latter defines a locally closed
embedding, $g:X\hookrightarrow G$, into the Grassmannian $G:={\rm Gr}(r,n)$ of
dimension~$r$ subspaces in a dimension $n$ vector space. Besides, $\F(N)\cong g^*\E$,
where $\E$ is the dual of the tautological bundle on the Grassmannian. Consider the
locally closed embedding
$$
X\hookrightarrow G\times \P^m
$$
and the closure of its image, $\bar X$. It remains to set
$\bar\F:=p_1^*\E\otimes p_2^*\OO_{\P^m}(-N)$, where $p_1$ and $p_2$
denote the projections from $\bar X$ to $G$ and $\P^m$, respectively. If $X$ is
smooth, the additional condition on $\bar X$ and $D$ can be achieved by the
Hironaka theorem.
\end{proof}

We also need the following result.

\begin{prop}\label{prop-redproj}
Let $X$ be an open subvariety in a smooth projective variety $\bar X$ such that the
complement  $D:=\bar X\smallsetminus X$ is a simple normal crossing divisor in
$\bar X$. Suppose $\bar \F$ is a locally free sheaf on $\bar X$ such that
$\F:=\bar\F|_X$. Then we have
$$
\F_{\pol}=\varinjlim_{n}\,\bar\F(nD)_{\pol}|_X,
$$
$$
\underline{\Pol}_{\,\bullet}(X,\F)=
\varinjlim_{n}~\underline{\Pol}_{\,\bullet}(\bar X,\bar\F(nD))|_X.
$$
\end{prop}
\begin{proof}
To show the first equality let us choose an arbitrary open subset $U$ in $X$ and
compare sections of both sheaves on $U$. By the Hironaka theorem,
there exists a smooth projective variety $\bar Y$ and a proper birational morphism
$\bar \pi:\bar Y\to \bar X$ such that $\bar\pi$ is an isomorphism over $U$ and the
following  divisors in $\bar Y$ are simple normal crossing:
$\widetilde{D}:=\bar\pi^{-1}(D)$, the closure $E$ of
$\bar\pi^{-1}(X\smallsetminus U)$ in $\bar Y$, and
$\bar\pi^{-1}(\bar X\smallsetminus U)=E\cup \widetilde{D}$. We set
$Y:=\bar\pi^{-1}(X)=\bar Y\smallsetminus\widetilde{D}$ and $\pi:=\bar\pi|_{Y}$.
Below, we use that for a variety $V$, a Cartier divisor $W\subset V$, and a coherent
sheaf $\Gc$ on $V$, one has that
$H^0(V\smallsetminus W,\Gc|_{V\smallsetminus W})=\varinjlim\limits_n H^0(V,\Gc(nW))$.
Thus, we have
$$
\F_{\pol}(U)=H^0(Y,\pi^!\F\otimes\OO_Y(\pi^{-1}(X\smallsetminus U))=
\varinjlim_n H^0(\bar Y,\bar\pi^!\bar\F\otimes\OO_{\bar Y}(E)
\otimes\OO_{\bar Y}(n\widetilde{D}))=
$$
$$
=\varinjlim_n H^0(\bar Y,\bar\pi^!\bar\F\otimes\OO_{\bar Y}(n\widetilde{D})
\otimes\OO_{\bar Y}(E\cup\widetilde{D}))=
\varinjlim_n H^0(\bar Y,\bar\pi^!\bar\F
\otimes\bar\pi^*\OO_{\bar X}(nD)
\otimes\OO_{\bar Y}(E\cup\widetilde{D}))=
$$
$$
=\varinjlim_n H^0(\bar Y,\bar\pi^!\left(\bar\F(nD)\right)
\otimes\OO_{\bar Y}(\bar\pi^{-1}(\bar X\smallsetminus U)))=
\varinjlim_n \bar\F(nD)_{\pol}(U),
$$
which proves the first equality.

For the second equality it is enough to show that for any irreducible
closed subvariety
$\bar Z\subset \bar X$ with non-empty $Z:=X\cap \bar Z$, we have
$$
\Pol_{Z}(\F|_Z)=\varinjlim_n\limits\Pol_{\bar Z}(\bar \F(nD)|_{\bar Z}).
$$
To prove the latter, recall that, by definition, for each $n$, the space
$\Pol_{\bar Z}(\bar \F(nD)|_{\bar Z})$
is the union in $\omega_{k(Z)}\otimes_{k(Z)}\F_{k(Z)}$ of the subspaces
$$
H^0(\bar V,\omega_{\bar V}(\bar W)\otimes_{\OO_{\bar V}}
\bar f^*\bar\F(nD)|_{\bar Z})
$$
over all collections $(\bar V,\bar f,\bar W)$, where $\bar V$ is a smooth variety,
$\bar f:\bar V\to \bar Z$ is a proper birational morphism, and $\bar W\subset \bar V$
is a simple normal crossing divisor in $\bar V$. Furthermore, by the Hironaka
theorem and Corollary~\ref{corol-directimage}, we may consider only collections
$(\bar V,\bar f,\bar W)$ such that
$C:=\bar f^{-1}(\bar Z\smallsetminus Z)=\bar f^{-1}(\bar Z\cap D)$ and $C\cup \bar W$
are simple normal crossing divisors in $\bar V$. On the other hand, by
Lemma~\ref{lemma-resolmorph} and Corollary~\ref{corol-directimage}, the group
$\Pol_Z(\F|_Z)$ is equal to the union in $\omega_{k(Z)}\otimes_{k(Z)}\F_{k(Z)}$ of
the groups
$$
H^0(V,\omega_{V}(W)\otimes_{\OO_{V}}f^*\F|_{Z})
$$
over all collections $(\bar V,\bar f,\bar W)$ as above, where $V:=\bar f^{-1}(Z)$,
$f:=\bar f|_V$, and $W:=\bar W\cap V$. Analogously to what was done for the first
equality, we have
$$
H^0(V,\omega_{V}(W)\otimes_{\OO_{V}}f^*\F|_{Z})=
\varinjlim_n H^0(\bar V,\omega_{\bar V}(\bar W)
\otimes_{\OO_{\bar V}}\bar f^*\bar\F|_{\bar Z}\otimes_{\OO_{\bar V}}\OO_{\bar V}(nC))=
$$
$$
=\varinjlim_n H^0(\bar V,\omega_{\bar V}(\bar W)\otimes_{\OO_{\bar V}}
\bar f^*\bar\F(nD)|_{\bar Z}) \,,
$$
which proves the second equality.
\end{proof}

Recall that for $X$, $\bar X$, $D$ as in Proposition~\ref{prop-redproj}
and for any $p\geqslant 0$, we have
$$
H^p(X,\F)=\varinjlim_n\,H^p(\bar X,\bar \F(nD))\,,
$$
because filtered direct limits commute with cohomology of complexes. Therefore,
combining Lemma~\ref{lemma-Orlov}, Proposition~\ref{prop-redproj}, and
Remark~\ref{rmk-openpolar}, we complete the reduction to the projective case and, thus,
finally prove Theorems~\ref{theor-polarresol} and \ref{theor-polarquasiis}.

\section{Further properties of polar chains}\label{section-ps}

\subsection{A relation with Rost's cycle modules}\label{subsection-cyclemod}

In the paper \cite{Rost} M.\,Rost introduced the so-called {\it cycle modules}.
Roughly speaking, a cycle module over a variety $X$ is a rule that, to each
field $(K,\varphi)$ over $X$, associates a graded abelian group $M(K,\varphi)$.
The groups $M(K,\varphi)$ should be equipped with pull-back, trace, and residue maps
with respect to fields over $X$, which satisfy various natural properties. Let
us consider the groups $M_q(\F)(K,\varphi)$ introduced in
Definition~\ref{defin-polarform} and set
$$
M(\F)(K,\varphi):=\mbox{$\bigoplus\limits_{q\geqslant
0}M_q(\F)(K,\varphi)$} \,.
$$
Proposition~\ref{prop-polarforms} is equivalent to saying that
we have in this way obtained a cycle module:

\begin{prop}\label{prop-cyclemod}
Let $\F$ be a locally free sheaf on a variety $X$. Then, the correspondence
$$
(K,\varphi)\mapsto M(\F)(K,\varphi)
$$
extends to a cycle module over the variety $X$, where $(K,\varphi)$
runs over all fields over $X$.
\end{prop}
Indeed, Proposition~\ref{prop-polarforms}$(i)$-$(vi)$ establishes the existence of
all the data needed to define a cycle module over $X$, while
Proposition~\ref{prop-polarforms}$(vii)$ together with general properties of the
pull-back, trace, and residue maps for logarithmic forms (see
Section~\ref{sect-prelim}) imply Rost's axioms obeyed by the data above.

Until the end of this subsection we assume that $X$ is smooth and projective.
Then by Proposition~\ref{prop-polarforms}$(viii)$, the cycle module $M(\OO_X)$
is the pull-back of the cycle module $M(\OO_{\rm pt})$ over a point
${\rm pt}=\Spec(k)$: for any field $(K,\varphi)$ over $X$, we have
$$
M(\OO_X)(K,\varphi)=M(\OO_{\rm pt})(K,\varphi_K),
$$
where $\varphi_K\!:\Spec(K)\to \Spec(k)$ is the canonical morphism. Recall that
Theorem~6.1 in \cite{Rost} asserts local exactness of the Gersten complex
associated with any cycle module over a point. This implies
Theorem~\ref{theor-polarresol} for this particular case. However, note that an
analogue of Theorem~\ref{theor-polarresol} is not true for an arbitrary cycle
module over $X$.

In order to prove Theorem~\ref{theor-polarquasiis} for $\Gc=\omega_X$
within this approach one uses one more cycle module over a point:
$$
H(K,\varphi_K):=\varinjlim_U\,\mbox{$\bigoplus\limits_{q\geqslant
0}H^q_{dR}(U)$}\,,
$$
where $U$ runs through all open subset in a smooth irreducible variety $V$ over
$k$ with $k(V)=K$ and $H^{\bullet}_{dR}$ denotes the algebraic de Rham
cohomology. This cycle module was previously considered by S.\,Bloch and
A.\,Ogus in~\cite{BO}, where the authors also proved local exactness of the
corresponding Gersten complex for any smooth $X$. Moreover, they constructed a
spectral sequence that converges to $H^{\bullet}_{dR}(X)$ with $E^1$-term being
formed by homogeneous summands of this Gersten complex.

Below we use some facts from mixed Hodge theory (see~\cite{Del}). One
has a decreasing Hodge filtration $F^pH$ on the cycle module $H$ with
$$
F^pH(K,\varphi_K):=\varinjlim_U\,\mbox{$\bigoplus\limits_{q\geqslant
0}F^{p+r}H^q_{dR}(U)$}\,,
$$
where $r:=\dim\,U$. An important fact is that all (higher) differentials in the
Bloch--Ogus spectral sequence are strict with respect to the Hodge filtration.
(By Lefschetz principle, one can assume that $k=\C$, in which case de Rham
cohomology are endowed with mixed Hodge structures, which form an abelian
category with morphisms being strict with respect to the Hodge filtration.)

For $X$ smooth and projective, application of $F^0$ to the Bloch--Ogus spectral
sequence kills all lines except one, which gives the complex
$\Pol_{\bullet}(X,\OO_X)$. (The reason is that for a smooth affine variety $U$
of dimension $r$, we have that $F^{r}H^q_{dR}(U)=0$ unless $q=r$, while
$F^{r}H^r_{dR}(U)=(\omega_{\bar U})_{\pol}(U)$, where $\bar U$ is any smooth
compactification of~$U$.) This proves again Theorem~\ref{theor-polarresol} for
$\F=\OO_X$ and, moreover, shows that
$$
F^dH^{p+d}_{dR}(X)=H^p(X,(\omega_X)_{\pol})\,.
$$
On the other hand, the degeneration of the Hodge spectral sequence implies that
$F^dH^{p+d}_{dR}(X)=H^p(X,(\omega_X))$ Altogether, this implies
Theorem~\ref{theor-polarquasiis} for $\Gc=\omega_X$. Thus, \mbox{cycle} modules lead to
another proof of Theorem~\ref{theor-main} for smooth projective $X$ and
$\F=\OO_X$, that is, the main theorem in \cite{KRT}.

\subsection{Algebraic cycles}

Polar chains give a new construction for the classes of algebraic cycles in
the groups $H^q(X,\Omega_X^q)$, where $X$ is a smooth variety. Namely, let $Z$ be an
irreducible sub\-variety of dimension $p$ in a smooth irreducible variety $X$ of
dimension $d$ and let $f:V\to X$ be a proper morphism birational onto its image with
smooth $V$ and $f(V)=Z$. The differential of $f$ defines a canonical section
$\alpha_Z$ of the following bundle on $V$:
$$
\Homc(\Tc^{p}_V,f^*\,\Tc^{p}_X)=\omega_V\otimes_{\OO_V}f^*\,\Tc^{p}_X \,,
$$
where $\Tc_X^{p}:=\wedge^{p}\Tc_X$. The regular section
$\alpha_Z$ is thus an element in
$$
\Pol_Z(\Tc_X^{p}|_Z)\subset \Pol_p(X,\Tc^{p}_X)\,.
$$
Obviously, $\partial\alpha_Z=0$. By Theorem~\ref{theor-main}, the homology class of
$\alpha_Z$ in the complex $\Pol_{\,\bullet}(X,\Tc^{p}_X)$ defines an element in
$H^{d-p}(X,\Omega^{d-p}_X)$, because
$\omega_X\otimes_{\OO_X}\Tc^{p}_X\cong \Omega^{d-p}_X$ and
$$
H_p(\Pol_{\,\bullet}(X,\Tc^p_X))\cong H^{d-p}(X,\Omega^{d-p}_X)\,.
$$
By linearity, this defines an element in $H^{d-p}(X,\Omega^{d-p}_X)$
for any algebraic $p$-cycle on $X$ with coefficients in $k$.

It follows that the class of an algebraic cycle $\sum_i c_iZ_i$ on $X$, $c_i\in k$,
is trivial in $H^{d-p}(X,\Omega^{d-p}_X)$ if and only if the corresponding polar
cycle is a polar boundary. The latter means that there exists a chain
$\beta\in\Pol_{p+1}(X,\Tc_X^{p})$ such that
$\partial\beta=\oplus\, (c_i\,\alpha_{Z_i})$ with~$\alpha_{Z_i}$ defined as above.
In explicit terms, there exists a collection of $(p+1)$-dimensional irreducible
subvarieties $Y_j\subset X$ and polar elements
$$
\beta_j\in \omega_{k(Y_j)}\otimes_{k(Y_j)}(\Tc^{p}_X|_{_{Y_j}})_{k(Y_j)}
$$
such that, for each $i$, we have that
$\sum_j\res_{Y_jZ_i}(\beta_j)=c_i\alpha_{Z_i}$,
where the sum is taken over~$Y_j$ such that $Z_i\subset Y_j$.

Furthermore, by Theorem~\ref{theor-main}, an arbitrary element in
$H^{d-p}(X,\Omega_X^{d-p})$ can be represented by a polar chain
$$
\gamma=\oplus\gamma_Z\in \Pol_{p}(X,\Tc^{p}_X)
$$
with polar elements $\gamma_Z\in\omega_{k(Z)}\otimes_{k(Z)}(\Tc^{p}_X|_{_Z})_{k(Z)}$.
Then, the class of $\gamma$ in $H^{d-p}(X,\Omega_X^{d-p})$ is represented by an
algebraic cycle if and only if $\gamma$ is homologous in the polar complex
$\Pol_{\,\bullet}(X,\Tc^{p}_X)$ to a chain $\gamma'=\oplus\gamma'_Z$, where each
$\gamma'_Z$ is a multiple of a distinguished element~$\alpha_Z$ in
$\omega_{k(Z)}\otimes_{k(Z)}(\Tc^{p}_X|_{_Z})_{k(Z)}$ defined above.

\subsection{The Cousin complex}

The classical Cousin problem consists in finding a rational (or meromorphic) section
of a locally free sheaf with a given principle part. Similarly, the polar complex
discussed in this paper corresponds to a problem of finding a logarithmic form (with
coefficients in a locally free sheaf) which possesses given residues at its first
order poles. In this subsection we expand on a relation between the polar complex and
the Cousin complex.

Let $X$ be a smooth irreducible quasi-projective variety of dimension $d$. For an
irreducible subvariety $Z\subset X$ and a sheaf of abelian groups $\Pc$ on $X$,
denote by $\gamma_Z\Pc$ the subsheaf of $\Pc$ that consists of sections of $\Pc$
supported on $Z$. Consider $\gamma_Z\Pc$ as a sheaf on $Z$. The functor $\gamma_Z$
from the sheaves on $X$ to the sheaves on $Z$ is left exact. Let $\eta_Z$ denote the
generic point of $Z$. The local cohomology groups $H^i_{\eta_Z}(X,\Pc)$ are defined
as stalks at $\eta_Z$ of the right derived sheaves $R^i\gamma_Z\Pc$:
$$
H^i_{\eta_Z}(X,\Pc):=(R^i\gamma_Z\Pc)_{\eta_Z}.
$$
To every sheaf $\Pc$ one can associate a {\it Cousin complex}\,
$\Cous_{\bullet}(X,\Pc)$,
see~\cite{Har}. We will use a homological notation
for it.
Then, the Cousin complex is a chain complex whose terms are defined as follows:
$$
\mbox{$\Cous_{p}(X,\Pc)=\bigoplus\limits_{Z\in
X_{(p)}}H^{d-p}_{\eta_Z}(X,\Pc)$}.
$$

Let $\Gc$ be a locally free sheaf on $X$. In this case, the groups
$H^i_{\eta_Z}(X,\Gc)$ vanish unless $i=d-p$, where $p=\dim\,Z$,
while for $i=d-p$, we have that
$$
H^{d-p}_{\eta_Z}(X,\Gc)=
{\rm E}_{\OO_{X,Z}}(((\omega^{-1}_X\otimes_{\OO_X}\Gc)|_Z
\otimes_{\OO_Z}{\omega_Z})_{k(Z)})\,,
$$
where ${\rm E}_{\OO_{X,Z}}(M)$ denotes the injective hull of an $\OO_{X,Z}$-module
$M$. Moreover, there is a canonical isomorphism (cf.~\cite{Har}):
\begin{equation}\label{eq-cousin}
H^{d-p}(X,\Gc)\cong H_p(\Cous_{\,\bullet}(X,\Gc))\,.
\end{equation}

Back to polar complexes, Theorem~\ref{theor-polarresol} provides a resolution
$\Pol_{\,\bullet}(X,\F)$ for the sheaf $(\omega_X\otimes_{\OO_X}\F)_{\pol}$\,. The
explicit form of $\Pol_{\,\bullet}(X,\F)$ as of a direct sum over irreducible
subvarieties allows one to show that for any $p$-dimensional irreducible subvariety
$Z$ in~$X$, the local cohomology groups
$H^i_{\eta_Z}(X,(\omega_X\otimes_{\OO_X}\F)_{\pol})$ vanish unless $i=d-p$, where
$p=\dim\,Z$, while for $i=d-p$, we have that
$$
H^{d-p}_{\eta_Z}(X,(\omega_X\otimes_{\OO_X}\F)_{\pol})=\Pol_{Z}(\F|_Z)\,.
$$
This implies that the Cousin complex of the sheaf
$(\omega_X\otimes_{\OO_X}\F)_{\pol}$ coincides with the polar complex of $\F$:
\begin{equation}\label{eq-cous-pol=pol}
\Cous_{\,\bullet}(X,(\omega_X\otimes_{\OO_X}\F)_{\pol})=\Pol_{\,\bullet}(X,\F)\,.
\end{equation}
An injective morphism of sheaves,
$(\omega_X\otimes_{\OO_X}\F)_{\pol}\hookrightarrow\omega_X\otimes_{\OO_X}\F$, induces
a morphism of the corresponding Cousin complexes,
$\Cous_{\,\bullet}(X,(\omega_X\otimes_{\OO_X}\F)_{\pol})\to
\Cous_{\,\bullet}(X,\omega_X\otimes_{\OO_X}\F)$. Thus, by equation~\eqref{eq-cous-pol=pol}, we have an injection,
\begin{equation}\label{eq-pol-in-cous}
\Pol_{\,\bullet}(X,\F)\hookrightarrow\Cous_{\,\bullet}(X,\omega_X\otimes_{\OO_X}\F)\,,
\end{equation}
which is assembled from the injective maps
$$
\Pol_{Z}(\F|_Z)\hookrightarrow (\omega_Z\otimes_{\OO_Z}\F|_Z)_{k(Z)}\hookrightarrow
{\rm E}_{\OO_{X,Z}}((\omega_Z\otimes_{\OO_Z}\F|_Z)_{k(Z)})\,.
$$
Equations~\eqref{eq-cousin} and \eqref{eq-cous-pol=pol} together with
Theorem~\ref{theor-main} imply now that \eqref{eq-pol-in-cous} is a quasiisomorphism
and one can say that, for a locally free sheaf, the Cousin complex is quasiisomorphic
to its own first order pole part.

As an example consider a smooth curve $X$ and
choose $\F=\OO_X$ (cf.\ Example~\ref{examp-curve}). Then,
$$
\Pol_{\,\bullet}(X,\omega_X)=\{0\to \Pol(X,\Omega_X^1)\to
\mbox{$\bigoplus\limits_{x\in X}$}k(x)\to 0\}\,,
$$
$$
\Cous_{\,\bullet}(X,\omega_X)=\{0\to \Omega^1_{k(X)}\to
\mbox{$\bigoplus\limits_{x\in
X}$}\Omega^1_{k(X)}/\Omega^1_{\OO_{X,x}}\to 0\}\,,
$$
where $\Pol(X,\Omega_X^1)$ consists of rational $1$-forms that have only first order
poles on $X$ and for each point $x\in X$, the first order pole part of
$\Omega^1_{k(X)}/\Omega^1_{\OO_{X,x}}$ is identified with $k(x)$ by the residue map.

\begin{rmk}
It might be possible to prove Theorem~\ref{theor-main} in the same way as one proves
that the Cousin complex $\Cous_{\,\bullet}(X,\Gc)$ computes cohomology groups
$H^{\bullet}(X,\Gc)$ for a locally free sheaf $\Gc$ on $X$ (cf.~\cite{Har}). This
could be achieved if, for each closed subset $i:Z\hookrightarrow X$, one would be
able to replace the ``topological'' functor $\gamma_Z$ relevant for $\Cous_{\,\bullet}$
by the ``algebro-geometric'' functor $i^!$ relevant for $\Pol_{\,\bullet}$\,.
\end{rmk}

\subsection{The adelic complex}

A.\,N.\,Parshin \cite{Par} introduced a certain complex $\A_{rat}(X,\F)^{\bullet}$
for a quasi-coherent sheaf $\F$ on a surface $X$, nowadays called the {\it rational
adelic complex}. Generalizations to higher dimensions were studied in a number of
papers, in particular, see~\cite{Par2}, \cite{Bei}, and \cite{Hub}. One of the main
properties of the rational adelic complex $\A_{rat}(X,\F)^{\bullet}$ is that it gives
a flabby resolution for any quasi-coherent sheaf $\F$ on $X$. In particular, there is
canonical isomorphism $H^p(X,\F)\cong H^p(\A_{rat}(X,\F)^{\bullet})$. For a smooth
complex algebraic variety, there are several analogies between the adelic resolution
and Dolbeault resolution (see, e.g.,~\cite{Yek}). Although any two resolutions can be
related with help of the canonical Godement construction, this is by no means
explicit. In this somewhat sketchy subsection we describe a construction of a chain of
quasiisomorphisms between the adelic and Dolbeault resolutions for smooth complex
projective varieties, such that each term of the chain is represented by an explicit
geometric construction.

A version of adelic complexes for sheaves of abelian groups was defined in~\cite{Gor}
(a different definition, appropriate to a more general situation, is given
in~\cite{BV}). Let us briefly describe the shape of this adelic complex. Suppose that
$X$ is an irreducible variety and $\Pc$ is a sheaf of abelian groups on $X$. A
non-degenerate flag on $X$ is a strictly decreasing chain of irreducible subvarieties
in $X$:
$$
Z_0\supsetneq Z_1\supsetneq\ldots\supsetneq Z_p   \;.
$$
Each term of the adelic complex $\A(X,\Pc)^{\bullet}$ is by definition  a certain
subgroup,
$$
\A(X,\Pc)^p\subset \prod_{(Z_0\ldots Z_p)}\Pc_{\eta_0}\,,
$$
in the direct product of stalks $\Pc_{\eta_0}$ of the sheaf $\Pc$, where the product
is taken over non-degenerate flags $(Z_0\ldots Z_p)$ in $X$, while each stalk
$\Pc_{\eta_0}$ is taken at the generic point $\eta_0\in Z_0$ of the largest elements
$Z_0$ in a flag. Thus, an element $f\in \A(X,\Pc)^p$ is a collection of elements
$f_{Z_0\ldots Z_p}\in \Pc_{\eta_0}$ subject to a certain condition, which we do not
specify explicitly here. The differential in the adelic complex is given by the
formula
$$
d(f)_{Z_0\ldots Z_{p+1}}=
\sum_{i=0}^{p+1}(-1)^if_{Z_0\ldots\hat{Z}_i\ldots Z_{p+1}}\in \Pc_{\eta_0}\,,
$$
where the hat means that we omit the corresponding element in a flag.

For example, if $X$ is a smooth curve, we have
$$
\A(X,\Pc)^{\bullet}=\{0\to \Pc_{\eta}\oplus\prod_{x\in X}\Pc_x\to
{\prod_{x\in X}}'\Pc_{\eta}\to 0\}\,,
$$
where $\eta$ denotes the generic point of $X$ and the symbol $\prod'$
means that we consider the set of collections
$$
f=(f_x)\in\prod_{x\in X}\Pc_{\eta}
$$
such that $f_x\in {\rm Im}(\Pc_x\to \Pc_{\eta})$ for almost all
$x\in X$. The latter condition explains the name ``adelic'' for the
complex $\A(X,\Pc)^{\bullet}$.

It is proved in~\cite{Gor} that for a smooth $X$ and a wide range of sheaves $\Pc$ on
$X$, the adelic complex $\A(X,\Pc)^{\bullet}$ gives a flabby resolution of $\Pc$.
This seems to extend to the case $\Pc=\Gc_{\pol}$\,, where $\Gc$ is a locally free
sheaf; to achieve this one should combine methods from~\cite{Gor} with the last part
of the proof in Section~\ref{section-polarresol}. Thus, for a smooth variety $X$, we
may assume that the following isomorphism holds:
$$
H^i(X,(\omega_X\otimes_{\OO_X}\F)_{\pol})\cong
H^i(\A(X,(\omega_X\otimes_{\OO_X}\F)_{\pol})^{\bullet})\,.
$$

On the other hand, it follows from the definitions of the rational adelic complex
$\A_{rat}(X,\Gc)^{\bullet}$ and the adelic complex $\A(X,\Pc)^{\bullet}$ with
$\Pc=\Gc_{\pol}$ that there is a canonical injective morphism of complexes
$$
\A(X,\Gc_{\pol})^{\bullet}\hookrightarrow
\A_{rat}(X,\Gc)^{\bullet},
$$
which agrees with the embedding of sheaves
$\Gc_{\pol}\hookrightarrow \Gc$. The above assumption implies that
this is a quasiisomorphism.

For an arbitrary sheaf of abelian groups $\Pc$, there is a morphism of complexes from
the adelic complex to the Cousin complex. Let us consider the case when
$\Pc=\Gc_{\pol}$ as above. Then, for each $p$, $0\leqslant p\leqslant d=\dim\,X$,
there is a surjective map
$$
\nu:\A(X,\Gc_{\pol})^p\to \Cous_{d-p}(X,\Gc_{\pol})=
\Pol_{d-p}(X,\omega^{-1}_X\otimes_{\OO_X}\Gc)
$$
defined by the formula
$$
\nu(f)_{Z}=(-1)^{\frac{p(p+1)}{2}}\sum_{(XZ_1\ldots Z_{p-1}Z)}
(\res_{Z_{p-1}Z}\,\circ\ldots\circ\,\res_{XZ_{1}})(f_{XZ_1\ldots Z_{p-1}Z})\in
\Pol_Z((\omega^{-1}_X\otimes_{\OO_X}\Gc)|_Z)\,,
$$
where $Z\subset X$ is an irreducible subvariety of codimension $p$ in $X$ and
$(X Z_1\ldots Z_{p-1}Z)$ are complete non-degenerate flags between $X$ and $Z$.
Here we use that $f_{XZ_1\ldots Z_{p-1}Z}$ belongs to
$\Pc_{\eta}=\Pol_X(\omega^{-1}_X\otimes_{\OO_X}\Gc)$,
where $\eta$ is the generic point of $X$. We have in fact
a surjective morphism of complexes
$$
\nu\!:\A(X,\Gc_{\pol})^{\bullet}
\twoheadrightarrow\Pol_{d-\bullet}(X,\omega^{-1}_X\otimes_{\OO_X}\Gc)\,.
$$
The above discussion, together with Theorem~\ref{theor-main}, implies that this
morphism is a quasiisomorphism.

On the other hand, if we choose $k=\C$, we can consider, on a smooth variety $X$ as
on any complex manifold, the following natural complexes constructed with help of the
Dolbeault $\bar\partial$-operator. Let $\Lambda^{0,p}(X,\Gc)$ be the space of
$C^\infty$ differential $(0,p)$-forms on $X$ with coefficients in the holomorphic
vector bundle corresponding to $\Gc$. Then, the $\bar\partial$-operator,
$\bar\partial\!:\Lambda^{0,p}(X,\Gc)\to\Lambda^{0,p+1}(X,\Gc)$, gives rise to a
complex $\Lambda^{0,\bullet}(X,\Gc)$, which computes the cohomology groups of the
sheaf associated with $\Gc$ in the analytic topology. One can also consider
differential forms with distributional coefficients, or currents, $D^{0,p}(X,\Gc)$,
which form a complex as well, $\bar\partial\!:D^{0,p}(X,\Gc)\to D^{0,p+1}(X,\Gc)$. A
well-known property of these complexes (see~\cite{GH}) is that the obvious embedding,
$$
\Lambda^{0,\bullet}(X,\Gc)\hookrightarrow D^{0,\bullet}(X,\Gc) \,,
$$
is in fact a quasiisomorphism. The elements of the space of currents $D^{0,p}(X,\Gc)$
are, by definition, certain linear functionals on
$\Lambda^{0,d-p}_c(X,\omega_X\otimes_{\OO_X}\Gc^\vee)$, the space of $C^\infty$
differential forms on $X$ with compact support, where $\Gc^\vee$ is dual to $\Gc$.
Recall now that a polar $r$-chain $\alpha\in\Pol_Z(X,\F)\subset\Pol_r(X,\F)$, where
$Z\subset X$ is an $r$-dimensional irreducible subvariety in $X$, can also be used to
define a linear functional on $\Lambda^{0,r}_c(X,\F^\vee)$: take a $(0,r)$-form
$u\in\Lambda^{0,r}_c(X,\F^\vee)$, restrict it to $Z$, and integrate the product of
$(2\pi i)^{-r}\alpha$ and  $u|_Z$ over $Z$ (it is important here that $\alpha$ has
only first order poles, for these are locally integrable singularities). This gives
us yet another embedding of complexes (considered also in~\cite{KR}):
$$
\Pol_{d-\bullet}(X,\omega_X^{-1}\otimes_{\OO_X}\Gc)
\hookrightarrow D^{0,\bullet}(X,\Gc) \,.
$$
Collecting all above together, we obtain the following chain of morphisms
of complexes:
\begin{equation}\label{eq-chain-qii}
\A_{rat}(X,\Gc)^{\bullet}\hookleftarrow \A(X,\Gc_{\pol})^{\bullet}\twoheadrightarrow
\Pol_{d-\bullet}(X,\omega^{-1}_X\otimes_{\OO_X}\Gc)\hookrightarrow
D^{0,\bullet}(X,\Gc)\hookleftarrow\Lambda^{0,\bullet}(X,\Gc) \,.
\end{equation}
In the case when $X$ is a smooth projective variety, the analytic and Zariski
cohomology groups of $\Gc$ are the same. Therefore, all complexes in
\eqref{eq-chain-qii} have the same cohomology groups and we obtain a chain of
quasiisomorphisms connecting the rational adelic complex $\A_{rat}(X,\Gc)^{\bullet}$
and the Dolbeault complex $\Lambda^{0,\bullet}(X,\Gc)$.

\end{document}